\documentclass{amsart}
\usepackage{hyperref}
\hypersetup{citecolor=blue,linkcolor=red}
\newtheorem{theorem}{Theorem}[section]

\newtheorem*{thmA1}{Theorem A1}
\newtheorem*{thmA2}{Theorem A2}
\newtheorem*{thmA3}{Theorem A3}

\newtheorem*{rem}{Remark}
\newtheorem{lemma}[theorem]{Lemma}

\newtheorem{fact}[theorem]{Fact}
\newtheorem{proposition}[theorem]{Proposition}

\newtheorem{corollary}[theorem]{Corollary}
\theoremstyle{definition}
\newtheorem{definition}[theorem]{Definition}

\theoremstyle{remark}
\newtheorem{remark}[theorem]{Remark}
\newtheorem*{noteadded}{Note Added in Proof}
\newcommand{\sgn}{\operatorname{sgn}}
\numberwithin{equation}{section}

\begin{document}

\title[Deforming a hypersurface by Gauss curvature and support function]
 {Deforming a hypersurface by Gauss curvature and support function}

\author[M.N. Ivaki]{Mohammad N. Ivaki}
\address{Institut f\"{u}r Diskrete Mathematik und Geometrie, Technische Universit\"{a}t Wien,
Wiedner Hauptstr. 8--10, 1040 Wien, Austria}
\curraddr{}
\email{mohammad.ivaki@tuwien.ac.at}

\dedicatory{}

\begin{abstract}
We study the motion of smooth, strictly convex bodies in $\mathbb{R}^n$ expanding in the direction of their normal vector field with speed depending on Gauss curvature and support function.
\end{abstract}
\maketitle
\section{Introduction}
The setting of this paper is $n$-dimensional Euclidean space, $\mathbb{R}^{n}$. A compact convex subset of $\mathbb{R}^{n}$ with non-empty interior is called a \emph{convex body}. The set of convex bodies in $\mathbb{R}^{n}$ is denoted by $\mathcal{K}^n$. Write $\mathcal{K}^n_{e}$ for the set of origin-symmetric convex bodies and $\mathcal{K}^n_{0}$ for the set of convex bodies whose interiors contain the origin. Also write $\mathcal{F}^n$, $\mathcal{F}^n_0$, and $\mathcal{F}^n_e$, respectively, for the set of smooth ($ C^{\infty}$-smooth), strictly convex bodies in $\mathcal{K}^n$, $\mathcal{K}^n_{0}$, and $\mathcal{K}^n_{e}$.

The unit ball of $\mathbb{R}^n$ is denoted by $B$ and its boundary is denoted by $\mathbb{S}^{n-1}$. We write $\nu:\partial K\to \mathbb{S}^{n-1}$ for the Gauss map of $\partial K$, the boundary of $K\in\mathcal{F}^n$. That is, at each point $x\in\partial K$, $\nu(x)$ is the unit outwards normal at $x$.

Assume that $\varphi$ is a positive, smooth function on $\mathbb{S}^{n-1}$. Let $F_0:M\to \mathbb{R}^n$ be a smooth parametrization of $\partial K_0$ where $K_0\in\mathcal{F}_0^n.$ In this paper, among other things, we study the long-time behavior of a family of convex bodies $\{K_t\}\subset \mathcal{F}_0^n$ given by smooth maps $F:M\times [0,T)\to \mathbb{R}^n$ that satisfies the initial value problem
\begin{equation}\label{e: flow0}
 \partial_{t}F(x,t)=\varphi(\nu(x,t))\frac{(F(x,t)\cdot \nu (x,t))^{2-p}}{\mathcal{K}(x,t)}\, \nu(x,t),~~
 F(\cdot,0)=F_{0}(\cdot).
\end{equation}
Here $F(M,t)=\partial K_t,$ and $\mathcal{K}(\cdot,t)$ is the Gauss curvature of $F(M,t).$ Moreover, $T$ is the maximal time for which the solution exists.

The support function of $K\in\mathcal{F}^n_0$ as a function on the unit sphere is defined by
\[h_K(u):= \nu^{-1}(u)\cdot u\]
for each $u\in\mathbb{S}^{n-1}$. All information about the hypersurface, except for parametrization, is contained
in the  support function. It easy to see that as $\{K_t\}$ moves according to (\ref{e: flow0}), then $h:\mathbb{S}^{n-1}\times [0,T)\to \mathbb{R},~ h(\cdot,t):=h_{K_t}(\cdot)$ evolves by
\begin{equation}\label{eq: flow4}
\partial_th(u,t)=\varphi(u)\frac{h^{2-p}}{\mathcal{K}}(u,t).
\end{equation}
A self-similar solution of this flow satisfies
\begin{align} \label{def: self similar}
h^{1-p} \det (\bar{\nabla}^2 h + \operatorname{Id}h)=\frac{c}{\varphi},
\end{align}
for some positive constant $c.$ Here $\bar{\nabla}$ is the covariant derivative on $\mathbb{S}^{n-1}$ endowed with an orthonormal frame.

When $p=2,~\varphi\equiv1$, flow (\ref{eq: flow4}), among other flows, was studied by Schn\"{u}rer \cite{Oliver 2006} in $\mathbb{R}^3,$ and by Gerhardt \cite{Gerhardt 2014} in higher dimensions. Both works rely on the reflection principle of Chow and Gulliver \cite{Bennett Chow and Robert Gulliver 1996}, and McCoy \cite{James A. McCoy 2003}. Their result is as follows: the normalized flow evolves any smooth strictly convex body in the $ C^{\infty}$-topology to an origin-centered ball. When $p=-n,~\varphi\equiv1$, the flow is a member of a family of flows, $p$-centro affine normal flows, which was introduced by Stancu \cite{Alina 2012}. In $\mathbb{R}^{2}$ and for $p=-2,~\varphi\equiv1$, a ``dual" flow to (\ref{eq: flow4}) (see Lemma \ref{lem: ev polar}) was studied by the author \cite{Ivaki 2013} with an application to the stability of the Busemann-Petty centroid inequality in the plane. For $p>2,~\varphi\equiv1$ and in $\mathbb{R}^n$, it follows from Chow-Gulliver \cite[Theorem 3.1]{Bennett Chow and Robert Gulliver 1996} (see also Tsai \cite[Example 1]{Tsai 2005}) that (\ref{eq: flow4}) evolves any smooth strictly convex body in the $ C^{1}$-topology to an origin-centered ball. See also Chow-Tsai \cite{Chow-Tsai 1996, Chow-Tsai 1997,Chow-Tsai 1998} for discussion of the expansion of convex hypersurfaces by non-homogeneous functions of principal curvatures and Gauss curvature. Moreover, in $\mathbb{R}^2$ the following theorems can be obtained by using Andrews' results. Let us set $\tilde{K}_t:=\left(V(B)/V(K_t)\right)^{1/n}K_t.$
\begin{thmA1}
Let $-2\leq p<\infty,~p\neq 1,~\varphi\equiv1$ and assume that $K_0\in\mathcal{F}_0^2$ satisfies $\int_{\mathbb{S}^{1}}\frac{u}{h_{K_0}(u)^{1-p}}d\sigma(u)=0$. There exists a unique solution $\{K_t\}\subset \mathcal{F}_0^2 $ of flow (\ref{eq: flow4}) such that $\{\tilde{K}_t\}$ converges in the $ C^{\infty}$-topology to the unit disk if $p>-2$ and to an origin-centered ellipse if $p=-2$.
\end{thmA1}
\begin{thmA2}
Let $-2< p<\infty,~p\neq 1.$ Let $\varphi$ be a positive, smooth, even function on $\mathbb{S}^1$ i.e., $\varphi(u)=\varphi(-u)$. Assume that $K_0\in\mathcal{F}_e^2$. There exists a unique solution $\{K_t\}\subset \mathcal{F}_e^2 $ of flow (\ref{eq: flow4}) such that $\{\tilde{K}_t\}$ converges in the $ C^{\infty}$-topology to an origin-symmetric strictly convex, smooth solution of (\ref{def: self similar}).
\end{thmA2}
\begin{thmA3}
Let $-2< p\leq -1,$ and  $K_0\in\mathcal{F}_0^2$ satisfy $\int_{\mathbb{S}^{1}}\frac{u}{\varphi(u)h_{K_0}(u)^{1-p}}d\sigma(u)=0$. Then there exists a unique solution $\{K_t\}\subset \mathcal{F}_0^2 $ of flow (\ref{eq: flow4}) such that $\{\tilde{K}_t\}$ converges in the $ C^{\infty}$-topology to a positive strictly convex, smooth solution of (\ref{def: self similar}).
\end{thmA3}
\begin{rem}
These theorems can be obtained from Andrews' results \cite{Andrews 1996, Andrews 1997,Andrews 1998,Andrews Ben 2000,Andrews 2003}: If $p<1$, then one needs Andrews' results about asymptotic behavior of shrinking flows by positive powers of curvature ($\partial_th=-\psi \mathcal{K}^{\frac{1}{1-p}}$), but when $p>1$ one needs Andrews' results on asymptotic behavior of expanding flows by negative powers of curvature ($\partial_th=\psi \mathcal{K}^{\frac{1}{1-p}}$). We observe that the evolution equation of $\psi \mathcal{K}^{\frac{1}{1-p}}$ in either case satisfies, up to a positive constant, (\ref{eq: flow4}) with $\varphi=\psi^{p-1}$. Existence of solutions to the Minkowski problem lets us reverse this procedure, provided $K_0$ satisfies the integral identity $\int_{\mathbb{S}^{1}}\frac{u}{\varphi h_{K_0}^{1-p}}d\sigma=0$. This argument is invalid if $n\ge3$ or $\int_{\mathbb{S}^{1}}\frac{u}{\varphi h_{K_0}^{1-p}}d\sigma\neq 0.$ See also S. Angenent, J.J.L. Vel\'{a}zquez, Y.-C. Lin, T.-S. Lin, C.-C. Poon, and D.-H. Tsai \cite{Angenent 1991,Angenent-V, Tsai 2004,Lin-Chu-Tsai 2010,Lin-Chu-Tsai 2012,Poon-Tsai 2010,Poon-Tsai 2014} for several beautiful results about the blow-up behavior of immersed, smooth, convex, closed plane curves with rotation index $m\ge 1$ evolving by (\ref{eq: flow4}).
\end{rem}
\begin{theorem}\label{thm: 1}
Let $n\ge 3,$ $p=-n$ and $\varphi\equiv1$. Assume that $K_0\in\mathcal{F}_0^n$ has its Santal\'{o} point at the origin, e.q., $\int_{\mathbb{S}^{n-1}}\frac{u}{h_{K_0}(u)^{n+1}}d\sigma(u)=0$. Then there exists a unique solution $\{K_t\}\subset \mathcal{F}_0^n $ of flow (\ref{eq: flow4}) such that $\{\tilde{K}_t\}$ converges in the $ C^{\infty}$-topology to an origin-centered ellipsoid.
\end{theorem}
As a corollary of this theorem, we prove an inequality of Lutwak \cite{Lutwak 1986} (stronger than the Blaschke-Santal\'{o} inequality). See Theorem \ref{thm: Lutwak 1986} for the statement.

The next theorem fills in the gap $p=1$ in the statement of Theorem A2.
\begin{theorem}\label{thm: 3a}
Let $p=1$, $\varphi$ be a positive, smooth, even function on $\mathbb{S}^1$ and $K_0\in\mathcal{F}_e^2$. Then there exists a unique solution $\{K_t\}\subset \mathcal{F}_0^2 $ of flow (\ref{eq: flow4}) such that $\{\tilde{K}_t\}$ converges in the $ C^{\infty}$-topology to an origin-symmetric strictly convex, smooth solution of (\ref{def: self similar}).
\end{theorem}
Let $F_0:M\to \mathbb{R}^n$ be a smooth parametrization of $\partial K_0$ where $K_0\in\mathcal{F}_0^n.$ Consider convex bodies $\{K_t\}\subset \mathcal{F}_0^n$ given by the smooth embeddings $F:M\times [0,T)\to \mathbb{R}^n$ that solve the initial value problem
\begin{equation}\label{e: flow2}
 \partial_{t}F(x,t)=-\frac{\mathcal{K}(x,t)}{(F(x,t)\cdot \nu (x,t))^{n}}\, \nu(x,t),~~
 F(\cdot,0)=F_{0}(\cdot).
\end{equation}
Then, as $K_t$ moves according to (\ref{e: flow2}), $h:\mathbb{S}^{n-1}\times [0,T)\to \mathbb{R},~h(\cdot,t):=h_{K_t}(\cdot)$ evolves by
\begin{equation}\label{eq: flow3}
\partial_th(u,t)=-\frac{\mathcal{K}}{h^{n}}(u,t).
\end{equation}
This flow was introduced by Stancu \cite{Alina 2012} ($p$ centro-affine normal flows for $p=\infty$).
We will prove the following theorem about the asymptotic behavior of flow (\ref{eq: flow3}).
\begin{theorem}\label{thm: 3}
Assume that $K_0\in\mathcal{F}_0^n$ has its centroid at the origin. Then there exists a unique solution $\{K_t\}\subset \mathcal{F}_0^n$ of flow (\ref{eq: flow3}) such that $\{\tilde{K}_t\}$ converges in the $C^{\infty}$-topology to an origin-centered ellipsoid.
\end{theorem}
Proofs of Theorems \ref{thm: 1} and \ref{thm: 3} result from finding a family of entropy functionals $B_p^{\varphi}$ (see Definition \ref{def: strong entropy}), a finding that was inspired by the definition of curvature image due to Petty \cite{Petty} and an inequality of Lutwak \cite{Lutwak 1986}; see Theorem \ref{thm: Lutwak 1986}. In the course of proving our main theorems, we also prove Theorems A1, A2 and A3 for a subsequence of times. In Section \ref{sec: ref to stab}, we give a convex-geometric argument to obtain the asymptotic shapes in Theorems \ref{thm: 1}, \ref{thm: 3}, A1, A2, and A3; the argument does not rely on uniform higher order regularity estimates for the normalized solutions, and it employs only entropy functionals $B_p^{\varphi}$. In Section \ref{sec: before application}, we will prove Theorem \ref{thm: 3a} and discuss the $C^{\infty}$ convergence in Theorems A1, A2, and A3. In Section \ref{sec: application}, we present a few applications of the flow (\ref{e: flow0}), such as a direct proof of Lutwak's inequality 1986.
\section{Background and notation}
\subsection{Differential Geometry}
The matrix of the radii of the curvature of $\partial K$ is denoted by $\mathfrak{r}=[\mathfrak{r}_{ij}]_{1\leq i,j\leq n-1}$ and the entries of $\mathfrak{r}$ are considered as functions on the unit sphere. They can be expressed in terms of the support function and its covariant derivatives as $\mathfrak{r}_{ij}:=\bar{\nabla}_i\bar{\nabla}_j h+h\bar{g}_{ij},$ where $[\bar{g}_{ij}]_{1\leq i,j\leq n-1}$ is the standard metric on $\mathbb{S}^{n-1}$ and $\bar{\nabla}$ is the standard Levi-Civita connection of $\mathbb{S}^{n-1}.$ The Gauss curvature of $\partial K$ is denoted by $\mathcal{K}$, and as a function on $\partial K$, it is also related to the support function of the convex body by \[\frac{1}{\mathcal{K}\circ\nu^{-1}}:=S_{n-1}=\det_{\bar{g}}[\bar{\nabla}_i\bar{\nabla}_jh+\bar{g}_{ij}h]:=\frac{\det [\mathfrak{r}_{ij}]}{\det{[\bar{g}_{ij}]}}.\]
In the sequel, for simplicity, we usually denote $\mathcal{K}\circ\nu^{-1}$ by $\mathcal{K}.$ The principal radii of curvature $\{\lambda_i\}_{1\le i\le n-1}$ are the eigenvalues of $[\mathfrak{r}_{ij}]$ with respect to $[\bar{g}_{ij}]$. Moreover, we write $[w_{ij}]$ for the second fundamental form of $\partial K$ and the principal curvatures are the eigenvalues of $[w_{ij}]$ with respect to $[g_{ij}]$ which we shall denote by $\{\kappa_i\}=\{1/\lambda_i\circ\nu\}.$
\subsection{Convex Geometry}
We will start by defining the polar body.\\
\textbf{\emph{Polar body:}}
The polar body, $K^{\ast}$, of convex body $K$ with the origin of $\mathbb{R}^n$ in its interior is the convex body defined as
\[K^{\ast}=\{x\in\mathbb{R}^n| x\cdot y\leq 1 \mbox{~for~all~}y\in K\}.\]
The Blaschke-Santal\'{o} inequality states that
\[\min_{x\in\operatorname{int} K}V(K)V((K-x)^{\ast})\leq \omega_n^2.\]
Equality holds exclusively for ellipsoids. The point for which the above minimum is achieved is called the Santal\'{o} point, and it will be denoted as $e_{-n}(K).$ In what follows, we will furnish all geometric quantities associated with $K^{\ast}$ with $~^{\ast}.$
\begin{theorem}\label{thm: eigenvalue relation body and it polar} Let $K\in \mathcal{F}^n_0.$ Suppose that $0<a\le h_K\le b<\infty$ and $0<c\le\kappa_i\le d<\infty.$  Then
$$c_1\le\kappa_i^{\ast}\leq c_2,$$
for $c_1,c_2>0$ depending only on $a,b,c,d.$
\end{theorem}
\begin{proof}
In general, parameterizing $\partial K$ as a graph over the unit sphere with the corresponding radial distance function $r:\mathbb{S}^{n-1}\to \mathbb{R}$, we can write
the metric $[g_{ij}]$ and its inverse $[g^{ij}]$, the second fundamental form $[w_{ij}]$, and $[\mathfrak{r}_{ij}^{\ast}]$ in terms of $r$ and whose spatial derivatives as follows:
\begin{enumerate}
  \item $\displaystyle g_{ij}=r^2\bar{g}_{ij}+\bar{\nabla}_ir\bar{\nabla}_jr;$
  \item $\displaystyle g^{ij}=\frac{1}{r^2}\left(\bar{g}^{ij}-\frac{\bar{\nabla}^ir\bar{\nabla}^jr}{r^2+|\bar{\nabla}r|^2}\right);$
  \item $\displaystyle w_{ij}=\frac{1}{\sqrt{r^2+|\bar{\nabla}r|^2}}\left(-r\bar{\nabla}_i\bar{\nabla}_jr+2\bar{\nabla}_ir\bar{\nabla}_jr+
  r^2\bar{g}_{ij}\right);$
  \item  $\displaystyle\mathfrak{r}^{\ast}_{ij}=\bar{\nabla}_i\bar{\nabla}_j\frac{1}{r}+\frac{1}{r}\bar{g}_{ij}=\frac{\sqrt{r^2+|\bar{\nabla}r|^2}}{r^3}w_{ij}.$
\end{enumerate}
The proofs of $(1)-(3)$ can be found in \cite{Zhu}, and $(4)$ follows from the fact that $\frac{1}{r}$ is the support function of $K^{\ast}$ (see also Oliker-Simon \cite[Indentities (7.6), (7.31)]{Oliker}).
We apply these formulas to $K^{\ast}.$ From $(2)$ and $(4)$ we get $[\mathfrak{r}_{ij}][g^{\ast ij}]=\frac{\sqrt{r^{\ast2}+|\bar{\nabla}r|^{\ast2}}}{r^{\ast3}}[w_{in}^{\ast}g^{\ast nj}]$. To prove the claim, we have only to consider points for which the gradient of the radial function $r^{\ast}$ does not vanish.
Around such a point, we introduce an orthonormal frame $\{e_1\cdots,e_{n-1}\}$ on $\mathbb{S}^{n-1}$ such that $e_1=\frac{\bar{\nabla}r{\ast}}{|\bar{\nabla}r{\ast}|}$. Then $\bar{\nabla} r{\ast}=(|\bar{\nabla} r{\ast}|,0,\cdots,0).$ Thus, in such a frame we may express $[\mathfrak{r}_{ij}][g^{\ast ij}]$ as follows:
 \begin{align}\label{e: det}
 [\mathfrak{r}_{ij}][g^{\ast ij}]=[\mathfrak{r}_{ij}]\left(
                        \begin{array}{cccc}
                          \frac{1}{r^{\ast2}+|\bar{\nabla}r|^{\ast2}} & 0 & \cdots & 0 \\
                          0 & \frac{1}{r^{\ast2}} & \cdots & 0\\
                          \vdots & \vdots & \ddots & \vdots\\
                          0 & \cdots & 0 & \frac{1}{r^{\ast2}}\\
                        \end{array}
                      \right):=AB.
  \end{align}
The eigenvalues of $A$ are $\{\lambda_i\}$, the eigenvalues of $B$ are $\{\frac{1}{r^{\ast2}+|\bar{\nabla}r|^{\ast2}}, \frac{1}{r^{\ast2}}\}$, and the eigenvalues of $AB$ are $\{\frac{\sqrt{r^{\ast2}+|\bar{\nabla}r|^{\ast2}}}{r^{\ast3}}\kappa_i^{\ast}\}$.
We may assume that $\kappa_1^{\ast}\leq \kappa_2^{\ast}\le\cdots\le \kappa_{n-1}^{\ast}$ and $\lambda_1\leq \lambda_2\le\cdots\le \lambda_{n-1}.$ It follows (for example, see Corollary III4.6 \cite{Bhatia 1997}) that the
eigenvalues of $AB$ are bounded above by $\frac{\lambda_{n-1}}{r^{\ast2}}.$ Thus, we get an upper bound on $\kappa_{n-1}^{\ast}$ in terms of $a,b,c.$ Moreover, from the identity $\frac{h_K^{n+1}}{\mathcal{K}}(x)\frac{h_{K^{\ast}}^{ n+1}}{\mathcal{K}^{\ast}}(x^{\ast})=1$
where $x\in\partial K,$ and $x^{\ast}\in\partial K^{\ast}$ satisfies $x\cdot x^{\ast}=1$ \footnote{This identity can be proved by taking the determinant of both sides of (\ref{e: det}). See also Hug \cite[Theorem 2.2]{Hug} for a proof of this identity for non-smooth, convex hypersurfaces.}, it follows that $l<\mathcal{K}^{\ast}$ for some positive finite number depending only on $a,b,d.$ Therefore, since $\kappa_{n-1}^{\ast}$ is bounded above, the lower bound on $\kappa_1^{\ast}$ follows.
\end{proof}
\begin{remark}
Explicit equality between the elementary symmetric functions of principal curvatures of $K^{\ast}$ and the principal radii of curvature of $K$ is given by Hug \cite[Theorem 5.1]{Hug2002} in a general setting in which the convex body might not be smooth. Moreover, \cite[Corollary 5.1]{Hug2002} deduces an inequality from which lower and upper bounds for the principal curvatures (and not only for their elementary symmetric functions) can be deduced.
\end{remark}
\begin{lemma}\label{lem: ev polar}
As $K_t$ evolve by (\ref{eq: flow4}), their polars $K_t^{\ast}$ evolve as follows:
\[\partial_th^{\ast}=-\varphi\left(\frac{h^{\ast}u+\bar{\nabla} h^{\ast}}{\sqrt{h^{\ast2}+|\bar{\nabla} h^{\ast}|^2}}\right)\left(\frac{(h^{\ast2}+|\bar{\nabla} h^{\ast}|^2)^{\frac{n+p}{2}}}{h^{\ast n}}\right)\mathcal{K}^{\ast},~~h^{\ast}(\cdot,t):=h_{K_t^{\ast}}(\cdot).\]
\end{lemma}
\begin{proof}
The proof is similar to the one in \cite[Theorem 2.2]{Ivaki Proc}.
\end{proof}
\textbf{\emph{Minkowski's mixed volume inequality and curvature function:}}
A convex body is said to be of class $ C^{k}_{+}$ for some $k\ge2$, if its boundary hypersurface is $k$-times continuously differentiable, in the sense of differential geometry, and if the Gauss map $\nu:\partial K\to \mathbb{S}^{n-1}$ is well-defined and a $ C^{k-1}$-diffeomorphism.

Let $K,L$ be two convex bodies and $0<a<\infty$. The Minkowski sum $K+aL$ is defined as $h_{K+aL}=h_K+ah_L$ and the mixed volume $V_1(K,L)$ of $K$ and $L$ is defined by
\[V_1(K,L)=\frac{1}{n}\lim_{a\to0^{+}}\frac{V(K+aL)-V(K)}{a}.\]
A fundamental fact is that corresponding to each convex body $K$, there is a unique Borel measure $S_K$ on the unit sphere such that
\[V_1(K,L)=\frac{1}{n}\int_{\mathbb{S}^{n-1}}h_L(u)dS_K(u)\]
for each convex body $L$. The measure $S_K$ is called the surface area measure of $K.$ Recall that if $K$ is $ C^2_+$, then $S_K$ is absolutely continuous with respect to $\sigma$, and the Radon-Nikon derivative $dS_K(u)/d\sigma(u)$ defined on $\mathbb{S}^{n-1}$ is the reciprocal Gauss curvature of $\partial K$ at the point of $\partial K$ whose outer normal is $u.$
For $K\in \mathcal{K}^{n}$,
\[V(K)=V_1(K,K)=\frac{1}{n}\int_{\mathbb{S}^{n-1}}h_K(u)dS_K(u).\]
Of significant importance in convex geometry is the Minkowski mixed volume inequality. Minkowski's mixed volume inequality states that for $K,L\in\mathcal{K}^n,$
\[V_1(K,L)^n\geq V(K)^{n-1}V(L).\]
Equality holds, if and only if $K$ and $L$ are homothetic.

A convex body $K$ is said to have a positive continuous curvature function $f_K$, defined on the unit sphere, provided that for every convex body $L$
\[V_1(K,L)=\frac{1}{n}\int_{\mathbb{S}^{n-1}}h_Lf_Kd\sigma,\]
where $\sigma$ is the spherical Lebesgue measure on $\mathbb{S}^{n-1}.$ A convex body can have at most one curvature function; see \cite[p.~115]{bon}.
If $K$ is of class $ C^2_+$, then the curvature function is the reciprocal Gauss curvature of $\partial K$ transplanted to $\mathbb{S}^{n-1}$ via the Gauss map.
\section{Entropy points and entropy functionals}\label{sec: entropy}
Write $\mathcal{S}^{+}$ for the set of positive, smooth functions on $\mathbb{S}^{n-1}$ and write $\mathcal{S}^{+}_e$ for the set of positive, smooth, even functions on the unit sphere. That is, $\varphi\in \mathcal{S}^{+}_e,$ if $\varphi\in \mathcal{S}^{+}$ and $\varphi(u)=\varphi(-u).$
\begin{lemma}\label{lem: 1}
There exists a unique point $e_p(K)\in \operatorname{int}K$ such that
\[
\begin{cases}
\min\limits_{x\in K}\int_{\mathbb{S}^{n-1}}(h_K(u)-x\cdot u)^pd\sigma=\int_{\mathbb{S}^{n-1}}(h_K(u)-e_p\cdot u)^pd\sigma &\mbox{if } 1<p<\infty\\
\max\limits_{x\in K}\int_{\mathbb{S}^{n-1}}(h_K(u)-x\cdot u)^pd\sigma=\int_{\mathbb{S}^{n-1}}(h_K(u)-e_p\cdot u)^pd\sigma &\mbox{if } 0<p<1\\
\min\limits_{x\in\operatorname{int}K}\int_{\mathbb{S}^{n-1}}-\log (h_K(u)-x\cdot u)d\sigma=\int_{\mathbb{S}^{n-1}}-\log (h_K(u)-e_0\cdot u)d\sigma & \mbox{if } p=0\\
\min\limits_{x\in\operatorname{int}K}\int_{\mathbb{S}^{n-1}}(h_K(u)-x\cdot u)^pd\sigma=\int_{\mathbb{S}^{n-1}}(h_K(u)-e_p\cdot u)^pd\sigma &\mbox{if } -n\leq p<0.
\end{cases}
\]
Additionally, $e_p$ is characterized by
$\int_{\mathbb{S}^{n-1}}\frac{u}{\left(h_K(u)-e_p(K)\cdot u\right)^{1-p}}d\sigma(u)=0.$

Moreover, let $-n\leq p\leq -n+1$ and $\varphi\in \mathcal{S}^{+}.$ Then there exists a unique point, $e_p^{\varphi}(K)$, in the interior of $K$ such that
\[\min\limits_{x\in\operatorname{int}K}\int_{\mathbb{S}^{n-1}}\frac{(h_K(u)-x\cdot u)^p}{\varphi(u)}d\sigma=\int_{\mathbb{S}^{n-1}}\frac{(h_K(u)-e_p^{\varphi}\cdot u)^p}{\varphi(u)}d\sigma.\]
Additionally, $e_p^{\varphi}$ is characterized by
$\int_{\mathbb{S}^{n-1}}\frac{u}{\varphi(u)\left(h_K(u)-e_p^{\varphi}(K)\cdot u\right)^{1-p}}d\sigma(u)=0.$
\end{lemma}
\begin{proof}
Existence and uniqueness of a point in $K$ for each of the above minimizations and maximizations follow from the strict concavity or strict convexity of the corresponding functional and compactness of $K$. The proof of $e_p\in\operatorname{int}K$ follows exactly the one given by Guan and Ni \cite[Lemmas 2.3, 2.4]{Guan NI 2013}. For completeness, we present it here. We shall consider the case $p\ne0.$ Suppose, on the contrary, that $e_p(K)$ is on the boundary of $K$ and $\nu$ is the outer normal at $e_p.$ By Busemann's theorem \cite[Theorem 1.12]{Busemann 2012}, there is a rectangular coordinate system $(y_1,\cdots, y_n)$ such that $e_p$ is the origin, $(0,\cdots,0,1)=\nu$, and the segment $[\vec{o}, -ty_n]$ is contained in $\operatorname{int}K$ for small $t>0$. In view of this fact, we may then assume that in the standard coordinate system of $\mathbb{R}^n$ one has $e_p=\vec{o}$, $\nu=(0,\cdots,1)$, and $K$ lies below the hyperplane $\nu^{\perp}.$ Take an arbitrary point $u^+=(u_1,\cdots,u_n)$, with $u_n\geq0$, and define $u^{-}=(u_1,\cdots,-u_n).$ For a fixed $u^+$ define $i(u^+)$, the point on the boundary of $K$ that $h_K(u^+)=u^+\cdot i(u^+).$ We have
\[h_K(u^-)\geq u^-\cdot i(u^+) \geq u^+\cdot i(u^+)=h_K(u^+).\]
Moreover, $h_K((0,\cdots,-1))>0$ and $h_K((0,\cdots,1))=0.$ So the above inequality must be strict for a set of positive measure.
Define $\bar{h}(u)=h_K(u)+su_n,$ and note that $\bar{h}$ is positive for all $u$, provided $(0,\cdots,-s)\in\operatorname{int}K,$ which is the case if $0<s<t$. Hence, we have
\begin{align*}
\sgn\left(\frac{d}{ds}\Bigg\vert_{s=0}\left(\int_{\mathbb{S}^{n-1}}\bar{h}^pd\sigma\right)\right)&=\sgn\left(p\int_{\mathbb{S}^{n-1}}h_K^{p-1}(u)(u\cdot\nu)d\sigma\right)\\
&=\sgn\left(p\int_{\{u_n>0\}\cap\mathbb{S}^{n-1}}\left(h_K^{p-1}(u^+)-h_K^{p-1}(u^-)\right)u_nd\sigma\right)\\
&=\sgn (p(1-p)).
\end{align*}
To prove that $e_p^{\varphi}\in\operatorname{int}K$, notice that when $-n\leq p\leq -n+1:$
\[\lim_{x\to\partial K}\int_{\mathbb{S}^{n-1}}\frac{(h_K(u)-x\cdot u)^p}{\varphi(u)}d\sigma=+\infty,\]
while by the Blaschke-Santal\'{o} inequality, the infimum is finite.
\end{proof}
\begin{remark}
In the sequel, we will always exclude case $p=1$, unless we are working with origin-symmetric bodies.
\end{remark}
\begin{definition}For $-n\leq p<\infty,~p\neq1$, $e_p(K)$ is the unique point in $\operatorname{int}K$ that satisfies
\[\int_{\mathbb{S}^{n-1}}\frac{u}{\left(h_K(u)-e_p\cdot u\right)^{1-p}}d\sigma(u)=0.\]
When $n\leq p\le -n+1$ and $\varphi\in \mathcal{S}^{+},$ $e_p^{\varphi}(K)$ is the unique point in $\operatorname{int}K$ that satisfies
\[\int_{\mathbb{S}^{n-1}}\frac{u}{\varphi(u)\left(h_K(u)-e_p\cdot u\right)^{1-p}}d\sigma(u)=0.\]
For $-n\leq p<\infty,~ \varphi\in \mathcal{S}^+_e$ and $K\in\mathcal{K}_e^n,$ we define the point $e_p^{\varphi}(K)$ to be the origin.
\end{definition}
\begin{remark}
In general, for $-n+1<p<\infty$, a minimizing or maximizing point of $\int_{\mathbb{S}^{n-1}}\frac{(h_K(u)-x\cdot u)^p}{\varphi(u)}d\sigma$ may fail to be in the interior of $K.$
\end{remark}
\begin{definition}
Let $-n\leq p<\infty.$
Since $K$ satisfies $\int\frac{u}{(h_K(u)-e_p\cdot u)^{1-p}(u)}d\sigma=0,$ the indefinite $\sigma$-integral of
$(h_K(u)-e_p\cdot u)^{p-1}$ satisfies the sufficiency condition of Minkowski's existence theorem in $\mathbb{R}^n.$ Hence, there exists a unique convex body (up to translations), denoted by $\Lambda_pK$, which has a surface area measure that satisfies
\begin{equation}\label{def: p curvature}
dS_{\Lambda_pK}=\left(\frac{V(K)}{\frac{1}{n}\int_{\mathbb{S}^{n-1}}h_{K-e_p}^pd\sigma}\right)\frac{1}{h_{K-e_p}^{1-p}}d\sigma,
\end{equation}
see Theorem 4 of \cite{ChYau}.
In addition, when $-n\leq p\leq-n+1~\mbox{and}~ \varphi\in \mathcal{S}^+$, or $-n\leq p<\infty,~\varphi\in \mathcal{S}^+_e~\mbox{and}~K\in\mathcal{K}_e^n,$ we define $\Lambda^{\varphi}_{p}K$ as a convex body with positive curvature function
\begin{equation}\label{def: p phi curvature}
f_{\Lambda_p^{\varphi}K}=\left(\frac{V(K)}{\frac{1}{n}\int_{\mathbb{S}^{n-1}}\frac{h_{K-e_p^{\varphi}}^p}{\varphi}d\sigma}\right)\frac{1}{\varphi h_{K-e_p^{\varphi}}^{1-p}}.
\end{equation}
\end{definition}
Notice that when $p=1$, $\Lambda_1^{\varphi}K$ is a ball. We point out that our definition of $\Lambda_p$ differs considerably from the usual definition of Lutwak (see \cite[p. 554]{Schneider}), but agrees with Petty's definition when $p=-n$ \cite{Petty}.
In the sequel, we will assume, after translation, that $\Lambda_pK$ and $\Lambda_p^{\varphi}K$ have the same centroids as $K.$ That is, $\operatorname{cent}(K)=\operatorname{cent}(\Lambda_pK)=\operatorname{cent}(\Lambda_p^{\varphi}K).$
\begin{remark}\label{rem: equality mixed minkowski}
From the definition of the mixed volume, we have $V_1(\Lambda_pK,K)=V(K).$ As a result, by the Minkowski mixed volume inequality
$ V(K)\geq V(\Lambda_pK).$ Equality holds if and only if $\Lambda_pK=K.$
Using Minkowski's mixed volume inequality once more, we get
\begin{equation}\label{ie: Minkowski}
V_1(K,\Lambda_pK)\geq V(\Lambda_pK).
\end{equation}
Here the equality holds if and only if $\Lambda_pK=K.$ Similarly if $-n\leq p\leq-n+1$ and $\varphi\in \mathcal{S}^+$, or $-n\leq p<\infty,~ \varphi\in \mathcal{S}^+_e$ and $K\in\mathcal{K}_e^n,$   we get
\begin{equation}\label{ie: Minkowski phi}
V_1(K,\Lambda_p^{\varphi}K)\geq V(\Lambda_p^{\varphi}K),
\end{equation}
and the equality holds if and only if $\Lambda_p^{\varphi}K=K.$
\end{remark}
\begin{remark}\label{rem: rem}
If $K\in \mathcal{F}^n$, then by definition $K$ is of class $ C^{\infty}_+$, so $h_K\in C^{\infty}$. In fact, by definition of the class $ C^{\infty}_+$, the Gauss map $\nu$ is a diffeomorphism of class $ C^{\infty}$ and so $h_K(u)=\nu^{-1}(u)\cdot u$ is $ C^{\infty}$. In this case, since $\Lambda_p^{\varphi} K,~\Lambda_p K$ are solutions of the Minkowski problem with positive $ C^{\infty}$ prescribed data, $\Lambda_p^{\varphi} K$ and $\Lambda_p K$ are of class $ C^{\infty}_{+}$; see Cheng-Yau \cite[Theorem 1]{ChYau}.
\end{remark}
\begin{definition}
\[
\mathcal{A}_p(K):=\begin{cases}
V(K)\left(\int_{\mathbb{S}^{n-1}}(h_K(u)-e_p\cdot u)^pd\sigma\right)^{-\frac{n}{p}} &\mbox{if } -n\leq p<\infty~\&~p\neq0\\
V(K)\exp\left(\frac{\int_{\mathbb{S}^{n-1}}-\log (h_K(u)-e_0\cdot u)d\sigma}{\omega_n}\right) & \mbox{if } p=0.
\end{cases}
\]
For $-n\leq p\leq-n+1,~\varphi\in \mathcal{S}^+$, or $-n\leq p<\infty,~p\neq0,~\varphi\in \mathcal{S}^+_e~\mbox{and}~K\in\mathcal{K}_e^n,$
\[
\mathcal{A}_p^{\varphi}(K):=
V(K)\left(\int_{\mathbb{S}^{n-1}}\frac{(h_K(u)-e_p\cdot u)^p}{\varphi(u)}d\sigma\right)^{-\frac{n}{p}}.
\]
For $p=0,~\varphi\in \mathcal{S}^+_e~\mbox{and}~K\in\mathcal{K}_e^n,$
\[
\mathcal{A}_0^{\varphi}(K):=
V(K)\exp\left(\frac{\int_{\mathbb{S}^{n-1}}-\frac{\log h_K}{\varphi}d\sigma}{\omega_n}\right).
\]
\end{definition}
Next, we introduce a new family of entropy functionals.
\begin{definition}\label{def: strong entropy}
$\mathcal{B}_p(K):=\frac{V(K)^{n-1}\mathcal{A}_p(K)}{V(\Lambda_pK)^{n-1}}$
and
$\mathcal{B}^{\varphi}_p(K):=\frac{V(K)^{n-1}\mathcal{A}_p^{\varphi}(K)}{V(\Lambda_p^{\varphi}K)^{n-1}}.$
\end{definition}
\begin{remark}
Note that we have $\mathcal{A}_p(K)\leq \mathcal{B}_p(K),$ and $\mathcal{A}^{\varphi}_p(K)\leq \mathcal{B}^{\varphi}_p(K).$ It will be shown that functionals $\frac{\log\mathcal{B}_p}{1-p},~\frac{\log\mathcal{B}_p^{\varphi}}{1-p}$, $p\neq 1$ are strictly increasing unless $K_t$ solves (in the Alexandrov sense) (\ref{def: self similar}). From this point of view, they will play roles in deducing the asymptotic shapes under the flows; see (\ref{eq: key asymp}).
\end{remark}
\section{Long-time existence}\label{sec: Long-time existence}
\begin{lemma}\label{lem: upper}
Let $\{K_t\}$ be a solution of (\ref{eq: flow4}) on $[0,t_0]$. If $c_1\leq h_{K_t}\leq c_2$ on $[0,t_0]$, then $\mathcal{K}\geq \frac{1}{a+b t^{-\frac{n-1}{n}}}$ on $(0,t_0],$ where $a$ and $b$ depend only on $c_1,c_2,p,\varphi.$ In particular, $\mathcal{K}\ge c_4$ on $[0,t_0]$ for some positive finite number that depends on the initial data, $c_1,c_2,p,\varphi$ and is independent of $t_0.$
\end{lemma}
\begin{proof}
Applying Tso's trick to the evolution equation for polar bodies, Lemma \ref{lem: ev polar}, as in the proof of \cite[Lemma 4.3]{Ivaki Proc} gives
$\mathcal{K}\geq \frac{1}{a+b t^{-\frac{n-1}{n}}}$
on $(0,t_0].$ It also follows from the proof of \cite[Lemma 4.3]{Ivaki Proc} that $a$ and $b$ depend only on $c_1,c_2,p,\varphi.$ The lower bound for $\mathcal{K}$ on $[0,\delta]$ for a small enough $\delta>0$ follows from the short-time existence of the flow. The lower bound for $\mathcal{K}$ on $[\delta,t_0]$ follows from the inequality $\mathcal{K}\geq \frac{1}{a+b \delta^{-\frac{n-1}{n}}}.$
\end{proof}
\begin{lemma}\label{lem: lower}
Let $\{K_t\}$ be a solution of (\ref{eq: flow4}) on $[0,t_0]$. If $0<c_2\leq h_{K_t}\leq c_1<\infty$ on $[0,t_0]$, then $\mathcal{K}\leq c_3<\infty$ on $[0,t_0].$ Here $c_3$ depends on the initial data, $c_1,c_2,p,\varphi$ and $t_0.$
\end{lemma}
\begin{proof}
We apply Tso's trick to the speed of (\ref{eq: flow4}) as in the proof of \cite[Lemma 4.1]{Ivaki Proc} to get \[\partial_t\frac{\varphi\frac{h^{2-p}}{\mathcal{K}}}{2c_1-h}\geq -c'\left(\frac{\varphi\frac{h^{2-p}}{\mathcal{K}}}{2c_1-h}\right)^2,\]
where $c'>0$ depends only on $c_1,c_2,p,\varphi.$
Therefore, \[\frac{\varphi\frac{h^{2-p}}{\mathcal{K}}}{2c_1-h}(t, u)\geq \frac{1}{c't+1/\min\limits_{u\in\mathbb{S}^{n-1}}\frac{\varphi\frac{h^{2-p}}{\mathcal{K}}}{2c_1-h}(0, u)}\geq \frac{1}{c't_0+1/\min\limits_{u\in\mathbb{S}^{n-1}}\frac{\varphi\frac{h^{2-p}}{\mathcal{K}}}{2c_1-h}(0, u)}.\]
The corresponding claim for the Gauss curvature follows.
\end{proof}

These last two lemmas are enough to establish the long-time existence of solutions to (\ref{eq: flow4}) when the initial body is in $\mathcal{F}^2_0.$
\begin{lemma}\label{lem: time singularity}
If $-2\leq p<2$ and $K_0\in\mathcal{F}^2_0$, then the lifespan of the solution to (\ref{eq: flow4}) is finite, and infinite when $p\ge 2$.
\end{lemma}
\begin{proof}
Let $-2\leq p<2.$ We can put a tiny disk centered at the origin inside $K_0.$ This disk flows to infinity in finite time, so by comparison principle $K_t$ cannot exist eternally.
For $p>2$, consider an origin-centered disk $B_R$, such that $K_0\subset B_R.$ Then $K_t\subset B_{R(t)},$ where $R(t)=\left((\max h_{K_0})^{p-2}+t(p-2)\max \varphi \right)^{\frac{1}{p-2}}.$ Thus, for any finite time $t_0$, $\{h_{K_t}\}$ remains uniformly bounded on $[0,t_0]$. So by Lemmas \ref{lem: upper}, \ref{lem: lower} the evolution equation (\ref{eq: flow4}) is uniformly parabolic on $[0,t_0]$ and bounds on higher derivatives of the support function follow. Therefore, we can extend the solution smoothly past time $t_0.$ When $p=2$ the argument is similar.
\end{proof}
\begin{proposition}\label{prop: expansion to infty}
Let $K_0\in\mathcal{F}^2_0$. Then the solution to (\ref{eq: flow4}) satisfies $\lim\limits_{t\to T}\max h_{K_t}=\infty.$
\end{proposition}
\begin{proof}
First, let $p> 2.$ In this case the flow exists on $[0,\infty).$ For this reason, we may insert a tiny disk inside of $K_0$ and use the comparison principle to prove the claim:
Consider an origin centered disk $B_r$, such that $K_0\supseteq B_r.$ Then $K_t\supseteq B_{r(t)},$ where $r(t)=\left((\min h_{K_0})^{p-2}+t(p-2)\min \varphi \right)^{\frac{1}{p-2}}$ and $B_{r(t)}$ expands to infinity as $t$ approaches $\infty$. When $p=2$ the argument is similar. Second, if $p<2$, then the flow exists only on a finite time interval. If $\max h_{K_t}<\infty$, then by Lemmas \ref{lem: upper} and \ref{lem: lower}, the evolution equation (\ref{eq: flow4}) is uniformly parabolic on $[0,T)$. Thus, the result of Krylov and Safonov \cite{Krylov-Safonov} and standard parabolic theory allow us to extend the solution smoothly past time $T$, contradicting its maximality.
\end{proof}
To prove the long-time existence of solutions to (\ref{eq: flow4}) when $n\ge3$ and $p=-n$, the next step is to obtain lower and upper bounds on the principal curvatures.
To this aim, we will use the evolution equation of $S_1^{\ast}:=\sum \lambda_i^{\ast}.$
\begin{lemma}\label{lem: upper and lower}
Let $n\ge3,~\varphi\equiv1$ and $p=-n$. Assume that $\{K_t\}$ is a solution of (\ref{eq: flow4}) on $[0,t_0].$ If $c_2\leq h_{K_t}\leq c_1$ and $c_4\leq\mathcal{K}\leq c_3$ on $[0,t_0],$ then
$$\frac{1}{C\left(1+t^{-(n-2)}\right)^{n-2}}\leq\kappa_i\leq C\left(1+t^{-(n-2)}\right)$$ on $(0,t_0],$ for some $C>0$ independent of $K_0$ and depending on $c_1,c_2,c_3,c_4,p.$ In particular, $\kappa_i$ are uniformly bounded above and stay uniformly away from zero on $[0,t_0].$
\end{lemma}
\begin{proof}
Since $c_2\leq h_{K_t}\leq c_1$, we have $c_2'\leq h^{\ast}(\cdot,t)=h_{K_t^{\ast}}\leq c_1'$. Moreover, since $c_4\leq\mathcal{K}\leq c_3$ on $[0,t_0]$, in view of the identity $\left(\frac{\mathcal{K}}{h^{n+1}}\right)(x)\left(\frac{\mathcal{K}^{\ast}}{h^{\ast n+1}}\right)(x^{\ast})=1,$ we get $c_4'\leq\mathcal{K}^{\ast}\leq c_3'$ on $[0,t_0].$ Now we calculate the evolution equation of $\mathfrak{r}_{ij}^{\ast}$ using Lemma \ref{lem: ev polar}. Set $\rho:=\frac{(h^{\ast2}+|\bar{\nabla} h^{\ast}|^2)^{\frac{n+p}{2}}}{h^{\ast n}}.$ Therefore,
\begin{align*}
\partial_t\mathfrak{r}_{ij}^{\ast}=&\rho S_{n-1}^{\ast-2}(S_{n-1}^{\ast})'_{kl}\bar{\nabla}_k\bar{\nabla}_l\mathfrak{r}_{ij}^{\ast}-2\rho
S_{n-1}^{\ast-3}\bar{\nabla}_iS_{n-1}^{\ast}\bar{\nabla}_jS_{n-1}^{\ast}\\
&+ \rho S_{n-1}^{\ast-2}(S_{n-1}^{\ast})''_{kl;mn}\bar{\nabla}_i\mathfrak{r}_{kl}^{\ast}\bar{\nabla}_j\mathfrak{r}_{mn}^{\ast}\\
&+(n-2)\rho S_{n-1}^{\ast-1}\bar{g}_{ij}-\rho S_{n-1}^{\ast-2}(S_{n-1}^{\ast})'_{kl}\mathfrak{r}_{ij}^{\ast}\bar{g}_{kl}\\
&-S_{n-1}^{\ast-1}\bar{\nabla}_i\bar{\nabla}_j \rho+S_{n-1}^{\ast-2}\bar{\nabla}_i\rho\bar{\nabla}_j S_{n-1}^{\ast}+ S_{n-1}^{\ast-2}\bar{\nabla}_j\rho\bar{\nabla}_i S_{n-1}^{\ast}.
\end{align*}
Thus, the evolution equation of $S_1^{\ast}$ satisfies
\begin{align*}
\partial_tS_1^{\ast}
=&\rho S_{n-1}^{\ast-2}(S_{n-1}^{\ast})'_{kl}\bar{\nabla}_k\bar{\nabla}_lS_1^{\ast}-2\rho
S_{n-1}^{\ast-3}|\bar{\nabla}S_{n-1}^{\ast}|^2\\
&+ \rho \bar{g}^{ij}S_{n-1}^{\ast-2}(S_{n-1}^{\ast})''_{kl;mn}\bar{\nabla}_i\mathfrak{r}_{kl}^{\ast}\bar{\nabla}_j\mathfrak{r}_{mn}^{\ast}\\
&+(n-2)(n-1)\rho S_{n-1}^{\ast-1}-\rho S_{n-1}^{\ast-2}S_1^{\ast}(S_{n-1}^{\ast})'_{kl}\bar{g}_{kl}\\
&-S_{n-1}^{\ast-1}\bar{\Delta} \rho+2S_{n-1}^{\ast-2}\bar{\nabla}\rho\cdot\bar{\nabla} S_{n-1}^{\ast}.
\end{align*}
\begin{description}
  \item[a] \emph{Estimating the terms on the first line:} The first term on the first line is a good term viewed as an elliptic operator that is non-positive at the point and direction at which the maximum of $S_1^{\ast}$ is achieved. The second term is a good non-positive term.
  \item[b] \emph{Estimating the term on the second line:} Concavity of $S_{n-1}^{\frac{1}{n-1}}$ gives
  \begin{equation}\label{ie: second derv}
  \left[(S_{n-1}^{\ast})''_{kl;mn}-\frac{n-2}{(n-1){S}_{n-1}^{\ast}}(S_{n-1}^{\ast})'_{kl}(S_{n-1}^{\ast})'_{mn}\right]
      \bar{\nabla}_{i}\mathfrak{r}_{kl}^{\ast}\bar{\nabla}_j\mathfrak{r}_{mn}^{\ast}\leq0.
  \end{equation}
  \item[c] \emph{Estimating the last term on the third line:} By Newton's inequality, we get
  \begin{equation}\label{e: line three}
  S_{n-1}^{\ast-2}S_1^{\ast}(S_{n-1}^{\ast})'_{kl}\bar{g}_{kl}=S_{n-1}^{\ast-2}S_1^{\ast}S_{n-2}^{\ast}\geq CS_{n-1}^{\ast-2}S_1^{\ast} S_{n-1}^{\ast\frac{n-3}{n-2}}S_1^{\ast\frac{1}{n-2}}.
      \end{equation}
   \item[d] \emph{Estimating the terms on the last line:} Since $p=-n,$ 
  \begin{equation}\label{e: line four}
  -S_{n-1}^{\ast-1}\bar{\Delta}\rho\leq CS_1^{\ast}+C,
  \end{equation}
  where we used boundedness of $|\bar{\nabla} h^{\ast}|$, $c_2'\leq h^{\ast}\leq c_1'$, and $c_4'\leq\mathcal{K}^{\ast}\leq c_3'$. To estimate the other term on the last line, we use Young's inequality and that $p=-n$: 
  \begin{equation}\label{e: line four, the other}
  |\bar{\nabla}\rho\cdot\bar{\nabla}S_{n-1}^{\ast}|\leq \frac{1}{2}(\varepsilon |\bar{\nabla}S_{n-1}^{\ast}|^2+\varepsilon^{-1}|\bar{\nabla}\rho|^2)\leq C(\varepsilon |\bar{\nabla}S_{n-1}^{\ast}|^2+\varepsilon^{-1}).
  \end{equation}
\end{description}
Combining inequalities (\ref{ie: second derv}), (\ref{e: line three}), (\ref{e: line four}), and (\ref{e: line four, the other}) with the lower and upper bounds on $S_{n-1}^{\ast},$  we obtain, for an $\varepsilon>0$ that is small enough, that
$$\partial_tS_1^{\ast}\leq C'\left(1+S_1^{\ast}-CS_1^{\ast\frac{n-1}{n-2}}\right).$$
This implies that $S_1^{\ast}\leq C(1+t^{-(n-2)})$ for some $C>0$ depending on $c_1,c_2,c_3,c_4,p.$ Therefore, in view of $\mathcal{K}^{\ast}\leq c_3',$  we get $$\frac{1}{C(1+t^{-(n-2)})}\leq\kappa_i^{\ast}\leq c_3'\left(C+Ct^{-(n-2)}\right)^{n-2}$$ on $(0,t_0].$ Finally Theorem \ref{thm: eigenvalue relation body and it polar} yields
$$\frac{1}{C\left(1+t^{-(n-2)}\right)^{n-2}}\leq\kappa_i\leq C\left(1+t^{-(n-2)}\right)$$ on $(0,t_0].$  Thus uniform lower and upper bounds for $\{\kappa_i\}$ on $[0,t_0]$ follow.
\end{proof}
\begin{lemma}\label{lem: time singularitya}
Let $n\ge3,~\varphi\equiv1$ and $p=-n$. Assume that $\{K_t\}$ is a solution of (\ref{eq: flow4}) with $K_0\in\mathcal{F}^n_0.$ Then the lifespan of the solution is finite.
\end{lemma}
\begin{proof}
Taking into account Lemmas \ref{lem: upper}, \ref{lem: lower} and \ref{lem: upper and lower}, the proof is similar to the one for Lemma \ref{lem: time singularity}.
\end{proof}
\begin{lemma}\label{lem: expansion to inftya}
Let $n\ge3,~\varphi\equiv1$ and $p=-n$. Assume that $\{K_t\}$ is a solution of (\ref{eq: flow4}) with $K_0\in\mathcal{F}^n_0$, then $\lim\limits_{t\to T}\max h_{K_t}=\infty.$
\end{lemma}
\begin{proof}
Considering Lemmas \ref{lem: upper}, \ref{lem: lower} and \ref{lem: upper and lower}, the proof is similar to the one for Proposition \ref{prop: expansion to infty}.
\end{proof}
\begin{proposition}\label{cor: volume to infty}
If $p=-n,~\varphi\equiv1$ and $K_0\in\mathcal{F}^n_0$, then $\lim\limits_{t\to T}V(K_t)=\infty.$
\end{proposition}
\begin{proof}
Since hypersurfaces are expanding, for each $t$ we can put a cone $C_t$ inside $K_t$ with height $\max h_{K_t}$ and a fixed origin-centered ball as the base. Lemma \ref{lem: expansion to inftya} shows that $\lim\limits_{t\to T}V(C_t)=\infty.$ The claim follows.
\end{proof}
\section{Monotonicity of entropies along the flow}
\begin{lemma}\label{lem: 8}The following statements hold:
\begin{itemize}
  \item Let $-n\leq p<\infty.$ If $\varphi\equiv1$ and $e_p(K_0)=\vec{o},$ then $e_p(K_t)=\vec{o}.$
  \item Let $-n\leq p\leq-n+1~\mbox{and}~ \varphi\in \mathcal{S}^+.$ If $e_p^{\varphi}(K_0)=\vec{o},$ then $e_p^{\varphi}(K_t)=\vec{o}.$
\end{itemize}
\end{lemma}
\begin{proof} We justify the first claim:
\[\frac{d}{dt}\int_{\mathbb{S}^{n-1}}\frac{u}{h_{K_t}^{1-p}(u)}d\sigma(u)=(p-1)\int_{\mathbb{S}^{n-1}}\frac{u}{\mathcal{K}(u)}d\sigma(u)=0.\]
Therefore, $e_p(K_t)=\vec{o}$ on $[0,T).$
\end{proof}
\begin{lemma} The following statements hold:
\begin{itemize}
  \item Let $-n\leq p<\infty$ and $\varphi\equiv1$. If $e_p(K_0)=\vec{o},$ then
\begin{align*}
\frac{d}{dt}V(\Lambda_pK_t)=&\frac{n}{n-1}\frac{V(\Lambda_pK_t)}{V(K_t)}\int_{\mathbb{S}^{n-1}}\frac{h_{K_t}^{2-p}}{\mathcal{K}^2}d\sigma-\frac{(1-p)n^2}{n-1}\frac{V(K_t)V_1(K_t,\Lambda_pK_t)}
{\int_{\mathbb{S}^{n-1}}h_{K_t}^p
d\sigma}\\ &-\frac{n^2p}{n-1}\frac{V(K_t)V(\Lambda_pK_t)}{\int_{\mathbb{S}^{n-1}}h_{K_t}^pd\sigma}.
\end{align*}
  \item If $-n\leq p\leq-n+1,~\varphi\in \mathcal{S}^+$ and $e_p^{\varphi}(K_0)=\vec{o},$ or $-n\leq p<\infty,~ \varphi\in \mathcal{S}^+_e~\mbox{and}~K_0\in\mathcal{F}_e^n,$ then
\begin{align*}
\frac{d}{dt}V(\Lambda_p^{\varphi}K_t)=&\frac{n}{n-1}\frac{V(\Lambda_p^{\varphi}K_t)}{V(K_t)}\int_{\mathbb{S}^{n-1}}\varphi\frac{h_{K_t}^{2-p}}{\mathcal{K}^2}d\sigma-\frac{(1-p)n^2}{n-1}
\frac{V(K_t)V_1(K_t,\Lambda_p^{\varphi}K_t)}
{\int_{\mathbb{S}^{n-1}}\frac{h_{K_t}^p}{\varphi}
d\sigma}\\ &-\frac{n^2p}{n-1}\frac{V(K_t)V(\Lambda_p^{\varphi}K_t)}{\int_{\mathbb{S}^{n-1}}\frac{h_{K_t}^p}{\varphi}d\sigma}.
\end{align*}
\end{itemize}
\end{lemma}
\begin{proof}
We will prove the first claim. Taking Lemma \ref{lem: 8} into account, computation is straightforward:
\begin{align*}
&\frac{d}{dt}V(\Lambda_pK_t)=\frac{1}{n-1}\int_{\mathbb{S}^{n-1}}h_{\Lambda_pK_t}\partial_t\left(\frac{V(K_t)}{\frac{1}{n}\int_{\mathbb{S}^{n-1}}h_{K_t}^pd\sigma}\frac{1}{h_{K_t}^{1-p}}\right)d\sigma\\
&=\frac{n}{n-1}\frac{V(\Lambda_pK_t)}{V(K_t)}\int_{\mathbb{S}^{n-1}}\frac{h_{K_t}^{2-p}}{\mathcal{K}^2}d\sigma-\frac{(1-p)n^2}{n-1}\frac{V(K_t)V_1(K_t,\Lambda_pK_t)}
{\int_{\mathbb{S}^{n-1}}h_{K_t}^p
d\sigma}\\
&-\frac{n^2p}{n-1}\frac{V(K_t)V(\Lambda_pK_t)}{\int_{\mathbb{S}^{n-1}}h_{K_t}^pd\sigma}.
\end{align*}
Note that Remark \ref{rem: rem} justifies that the taking time-derivative of the Gauss curvature of $ \Lambda_p^{\varphi}K_t$ is legitimate.
\end{proof}
\begin{lemma}\label{lem: monotonicity} We have for
\begin{itemize}
  \item $-n\leq p<\infty$ and $\varphi\equiv1:$ $\frac{d}{dt}\mathcal{A}_p(K_t)\geq 0$, and if $e_p(K_0)=\vec{o}$ and $p\neq 1$ then $\frac{d}{dt}\frac{\log \mathcal{B}_p(K_t)}{1-p}\geq 0.$
  \item $-n\leq p\leq-n+1~\mbox{and}~ \varphi\in \mathcal{S}^+:$ $\frac{d}{dt}\mathcal{A}_p^{\varphi}(K_t)\geq 0$, and if $e_p^{\varphi}(K_0)=\vec{o}$ then $\frac{d}{dt}\mathcal{B}_p^{\varphi}(K_t)\geq 0.$
  \item $-n\leq p<\infty, \varphi\in \mathcal{S}^+_e$ and $K_0\in\mathcal{F}_e^n:$ $\frac{d}{dt}\mathcal{A}_p^{\varphi}(K_t)\geq 0$, and if $p\neq1$ then $\frac{d}{dt}\frac{\log \mathcal{B}_p^{\varphi}(K_t)}{1-p}\geq 0.$
\end{itemize}
\end{lemma}
\begin{proof}
We prove the claims for $\mathcal{B}_p$ and $\mathcal{A}_p$ and $p\ne 0.$
Monotonicity of $\mathcal{A}_p(K_t)$ follows from the H\"{o}lder inequality:
\begin{align*}
&\frac{d}{dt}\mathcal{A}_p(K_t)\\
&=\frac{\biggl(\int_{\mathbb{S}^{n-1}}\frac{h_{K_t-e_p(K_t)}^{2-p}}{\mathcal{K}^2}d\sigma
\int_{\mathbb{S}^{n-1}}h_{K_t-e_p(K_t)}^pd\sigma -n^2V(K_t)^{2}\biggr)}{\left(\int_{\mathbb{S}^{n-1}}h_{K_t-e_p(K_t)}^pd\sigma\right)^{\frac{n}{p}+1}}\\
&+\frac{nV(K_t)}{\left(\int_{\mathbb{S}^{n-1}}h_{K_t-e_p(K_t)}^pd\sigma\right)^{\frac{n}{p}+1}}\left(\int_{\mathbb{S}^{n-1}}
\frac{u}{h^{1-p}_{K_t-e_p(K_t)}(u)}d\sigma\right)\frac{d}{dt}e_p(K_t)\\
&=\frac{\biggl(\int_{\mathbb{S}^{n-1}}\frac{h_{K_t-e_p(K_t)}^{2-p}}{\mathcal{K}^2}d\sigma
\int_{\mathbb{S}^{n-1}}h_{K_t-e_p(K_t)}^pd\sigma -\left(\int_{\mathbb{S}^{n-1}}\frac{h_{K_t-e_p(K_t)}}{\mathcal{K}}d\sigma\right)^2\biggr)}{\left(\int_{\mathbb{S}^{n-1}}h_{K_t-e_p(K_t)}^pd\sigma\right)^{\frac{n}{p}+1}}\\
&\geq0.
\end{align*}
Here we used the inverse function theorem to justify that $\frac{d}{dt}e_p(K_t)$ exists.

\noindent  Monotonicity of $\mathcal{B}_p(K_t)$ follows from Lemma \ref{lem: 8} and inequality (\ref{ie: Minkowski}):
\begin{multline}\label{eq: key asymp}
\frac{d}{dt}\mathcal{B}_p(K_t)
=\frac{1}{V(\Lambda_pK_t)^{n}}\Biggl[nV(K_t)^{n-1}V(\Lambda_pK_t)\int_{\mathbb{S}^{n-1}}\frac{h_{K_t}^{2-p}}{\mathcal{K}^2}d\sigma
\left(\int_{\mathbb{S}^{n-1}}h_{K_t}^pd\sigma\right)^{-\frac{n}{p}}\\ -n^2V(K_t)^{n+1}V(\Lambda_pK_t)\left(\int_{\mathbb{S}^{n-1}}h_{K_t}^pd\sigma\right)^{-\frac{n}{p}-1}\\
-nV(K_t)^{n-1}V(\Lambda_pK_t)\int_{\mathbb{S}^{n-1}}\frac{h_{K_t}^{2-p}}{\mathcal{K}^2}d\sigma\left(\int_{\mathbb{S}^{n-1}}h_{K_t}^pd\sigma\right)^{-\frac{n}{p}}\\
+(1-p)n^2V(K_t)^{n+1}V_1(K_t,\Lambda_pK_t)\left(\int_{\mathbb{S}^{n-1}}h_{K_t}^pd\sigma\right)^{-\frac{n}{p}-1}\\
+n^2pV(K_t)^{n+1}V(\Lambda_pK_t)\left(\int_{\mathbb{S}^{n-1}}h_{K_t}^pd\sigma\right)^{-\frac{n}{p}-1}\Biggr]\\
=\frac{n^2(1-p)V(K_t)^{n+1}\left(\int_{\mathbb{S}^{n-1}}h_{K_t}^pd\sigma\right)^{-\frac{n}{p}-1}}{V(\Lambda_pK_t)^{n-1}}
\biggl(\frac{V_1(K_t,\Lambda_pK_t)}{V(\Lambda_pK_t)}-1\biggr)\ge 0.
\end{multline}
\end{proof}
\section{Bounding extrinsic diameter}
\begin{lemma}\label{lem: bound diam}
Fix $0<a<\infty.$ Suppose $V(K)=\omega_n$ and
\[
\begin{cases}
\int_{\mathbb{S}^{n-1}}(h_K(u)-e_p\cdot u)^pd\sigma\leq a &\mbox{if } p\in (0,\infty)\\
\exp\left(\frac{1}{\omega_n}\int_{\mathbb{S}^{n-1}}-\log (h_K(u)-e_{0}\cdot u)d\sigma\right)\geq a& \mbox{if } p=0\\
\int_{\mathbb{S}^{n-1}}(h_K(u)-e_p(K)\cdot u)^pd\sigma\geq a &\mbox{if } p\in (-n,0).
\end{cases}
\]
Then the extrinsic diameter  of $K$, $d(K)$, is bounded above by a positive number independent of $K.$ The same statement also holds for the following cases:
\begin{itemize}
  \item when $-n<p\leq-n+1,~\varphi\in \mathcal{S}^+$ and
\[\int_{\mathbb{S}^{n-1}}\frac{(h_K(u)-e_p^{\varphi}(K)\cdot u)^pd\sigma}{\varphi(u)}\geq a,\]
  \item when $\varphi\in \mathcal{S}^+_e,~K\in\mathcal{K}_e^n$ and
\[
\begin{cases}
\int_{\mathbb{S}^{n-1}}\frac{h_K^p}{\varphi}d\sigma\leq a&\mbox{if } p\in (0,\infty)\\
\exp\left(\frac{1}{\omega_n}\int_{\mathbb{S}^{n-1}}-\frac{\log h_K}{\varphi}d\sigma\right)\geq a& \mbox{if } p=0\\
\int_{\mathbb{S}^{n-1}}\frac{h_K^p}{\varphi}d\sigma\geq a &\mbox{if } p\in (-n,0).
\end{cases}
\]
\end{itemize}
\end{lemma}
\begin{proof}
We prove only the first set of claims. The proof of Guan and Ni \cite[Corollary 2.5]{Guan NI 2013} extends to the interval $p\in(-1,\infty).$ For $-n<p<0,$ we argue as follows \cite[p. 58]{Chou Wang}: Suppose, on the contrary, that there is a sequence of convex bodies $\{K_i\}$ satisfying the uniform lower bound, but $d(K_i)\to\infty.$ Since the above inequalities are invariant under any translation, we may assume without loss of generality that $K_i$ are centered at the origin. Let $E_i$ denote John's ellipsoid of $K_i$. That is, $\frac{1}{n}E_i\subset K_i\subseteq E_i.$ Therefore, $\frac{h_{E_i}}{n}<h_{K_i}\leq h_{E_i}.$ For any fixed $\varepsilon>0$, we decompose $\mathbb{S}^{n-1}$ into three sets as follows:
\[S_1:=\mathbb{S}^{n-1}\cap \{h_{E_i}<\varepsilon\},~~S_2:=\mathbb{S}^{n-1}\cap \{\varepsilon<h_{E_i}<\frac{1}{\varepsilon}\},~~\&~S_3:=\mathbb{S}^{n-1}\cap \{h_{E_i}>\frac{1}{\varepsilon}\}.\]
On the one hand,
\[a\leq\int_{\mathbb{S}^{n-1}}(h_{K_i}(u)-e_p(K_i)\cdot u)^pd\sigma\leq \int_{\mathbb{S}^{n-1}}h_{K_i}^pd\sigma< \int_{\mathbb{S}^{n-1}}(\frac{h_{E_i}}{n})^pd\sigma.\]
On the other hand, as $d(K_i)\to\infty:$
\begin{align*}
\int_{S_1}(\frac{h_{E_i}}{n})^pd\sigma&\leq \left(\int_{S_1}(\frac{h_{E_i}}{n})^{-n}d\sigma\right)^{\frac{-p}{n}}|S_1|^{\frac{p+n}{n}}\\
&\leq c_1|S_1|^{\frac{p+n}{n}}\to 0,
\end{align*}
\[|S_2|\to 0.\]
Also, we have
\begin{align*}
\int_{S_3}(\frac{h_{E_i}}{n})^pd\sigma&\leq \int_{S_3}(\frac{1}{n\varepsilon})^pd\sigma=(\frac{1}{n\varepsilon})^p|S_3|\leq c_3\varepsilon^{-p}.
\end{align*}
Therefore, for any $\varepsilon>0$ we get
\[a\leq o(1)+c_3\varepsilon^{-p}.\]
Sending $\varepsilon\to0$, we reach a contradiction.
\end{proof}
\section{Continuity of entropy map and entropy functional}
\begin{theorem}[Continuity of entropy map]\label{thm: cont entropy points}
The following maps are continuous.
\begin{itemize}
  \item $e_p:(\mathcal{K}^n,d_{\mathcal{H}})\to \mathbb{R}^n,$ for $-n\le p<\infty,$
  \item $e_p^{\varphi}: (\mathcal{K}^n,d_{\mathcal{H}})\to \mathbb{R}^n,$ for $-n\leq p\leq-n+1~\mbox{and}~ \varphi\in \mathcal{S}^+.$
\end{itemize}
Here $d_{\mathcal{H}}$ denotes the Hausdorff distance.
\end{theorem}
\begin{proof}
We address the case $p\in [-n,1),~p\neq 0$ and the statement for $e_p.$ Suppose $p\in (0,1).$
Let $\{K_i\}$ be a family of convex bodies that converges to $K_{\infty}$ as $i$ approaches $\infty.$ Suppose that for a subsequence $i_j$ that $\lim\limits_{j\to\infty}e_p(K_{i_j})=q\neq e_p(K_{\infty}),$ where $q\in K_{\infty}.$ Note that $e_p(K_{\infty})\in \operatorname{int} K_{\infty}$ implies $e_p(K_{\infty})\in \operatorname{int} K_{i}$ if $i$ large is enough. Thus, for any $j$ that is large enough, we have
\[\int_{\mathbb{S}^{n-1}}(h_{K_{i_j}}(u)-e_p(K_{i_j})\cdot u)^pd\sigma > \int_{\mathbb{S}^{n-1}}(h_{K_{i_j}}(u)-e_p(K_{\infty})\cdot u)^pd\sigma. \]
Taking the limit from both sides, we get
\[\int_{\mathbb{S}^{n-1}}(h_{K_{\infty}}(u)-q\cdot u)^pd\sigma \geq \int_{\mathbb{S}^{n-1}}(h_{K_{\infty}}(u)-e_p(K_{\infty})\cdot u)^pd\sigma.\]
This contradicts that $e_p(K_{\infty})$ is the unique maximizer of $\int_{\mathbb{S}^{n-1}}(h_{K_{\infty}}(u)-x\cdot u)^pd\sigma$ on $K_{\infty}$. Now we consider the case $-n\leq p<0.$
Note that
\begin{align*}
&\limsup_{j\to\infty}\int_{\mathbb{S}^{n-1}}(h_{K_{i_j}}(u)-e_p(K_{i_j})\cdot u)^pd\sigma\\
&\leq \limsup_{j\to\infty}\int_{\mathbb{S}^{n-1}}(h_{K_{i_j}}(u)-e_p(K_{\infty})\cdot u)^pd\sigma=\int_{\mathbb{S}^{n-1}}(h_{K_{\infty}}(u)-e_p(K_{\infty})\cdot u)^pd\sigma.
\end{align*}
On the other hand, using Fatou's lemma we get
\[\liminf_{j\to\infty}\int_{\mathbb{S}^{n-1}}(h_{K_{i_j}}(u)-e_p(K_{i_j})\cdot u)^pd\sigma\geq\int_{\mathbb{S}^{n-1}}(h_{K_{\infty}}(u)-q\cdot u)^pd\sigma.\]
Thus, we have
\[\int_{\mathbb{S}^{n-1}}(h_{K_{\infty}}(u)-e_p(K_{\infty})\cdot u)^pd\sigma\geq\int_{\mathbb{S}^{n-1}}(h_{K_{\infty}}(u)-q\cdot u)^pd\sigma.\]
This is a contradiction.
\end{proof}
\begin{remark}
For $\varphi\equiv 1$ and $p=-n$, Theorem \ref{thm: cont entropy points} was proved by Petty \cite[Lemma 2.2]{Petty 1975}.
\end{remark}
\begin{theorem}[Continuity of $\mathcal{A}_p$]\label{thm: cont ent func}
The following functionals are continuous.
\begin{itemize}
  \item $\mathcal{A}_p:(\mathcal{K}^n,d_{\mathcal{H}})\to \mathbb{R}^n,$ for $-n\le p<\infty,$
  \item $\mathcal{A}_p^{\varphi}: (\mathcal{K}^n,d_{\mathcal{H}})\to \mathbb{R}^n,$ for $-n\leq p\leq-n+1~\mbox{and}~ \varphi\in \mathcal{S}^+$,
  \item $\mathcal{A}_p^{\varphi}: (\mathcal{K}_e^n,d_{\mathcal{H}})\to \mathbb{R}^n,$ for $-n\leq p<\infty~\mbox{and}~ \varphi\in \mathcal{S}_e^+.$
\end{itemize}
\end{theorem}
\begin{proof}
The first two claims follows from the continuity of $e_p(\cdot),~e_p^{\varphi}(\cdot) $, Theorem \ref{thm: cont entropy points}, and that $e_p,~e_p^{\varphi}$ are interior points. The last claim is trivial in view of our agreement that $e_p^{\varphi}(K)=\vec{o}$ whenever $K\in\mathcal{K}_e^n.$
\end{proof}
\begin{theorem} \label{thm: uniform lower bound on the support function}
 Fix $p$ and $0<a<\infty.$ Define the entropy class $S_{p,a}$ to be the set of all convex bodies $K$ such that $V(K)=\omega_n$ and
\[
\begin{cases}
e_p(K)=\vec{o},~\int_{\mathbb{S}^{n-1}}h_K^pd\sigma\leq a &\mbox{if } p\in (0,\infty)\\
e_p(K)=\vec{o},~\exp\left(\frac{1}{\omega_n}\int_{\mathbb{S}^{n-1}}-\log h_Kd\sigma\right)\geq a& \mbox{if } p=0\\
e_p(K)=\vec{o},~\int_{\mathbb{S}^{n-1}}h_K^pd\sigma\geq a>0 &\mbox{if } p\in (-n,0).
\end{cases}
\]
Then there exist $0<r,R<\infty$ depending only on $n,p,a$ such that for any $K\in S_{p,a}$ we have $r \leq h_{K}\leq R.$ Additionally, similar conclusions hold for the following sets:
\begin{itemize}
  \item when $\varphi\in \mathcal{S}^+$ and $-n<p\leq-n+1$, define the entropy class $S_{\varphi,p,a}$ to be the set of all convex bodies $K$ such that $V(K)=\omega_n$ and
\[e_p^{\varphi}~(K)=\vec{o},~\int_{\mathbb{S}^{n-1}}\frac{h_K^p}{\varphi}d\sigma\geq a,\]
  \item when $\varphi\in \mathcal{S}^+_e,$ define $S_{e,\varphi,p,a}$ to be the set of all origin-symmetric convex bodies $K$ such that $V(K)=\omega_n$ and
\[
\begin{cases}
\int_{\mathbb{S}^{n-1}}\frac{h_K^p}{\varphi}d\sigma\leq a &\mbox{if } p\in (0,\infty)\\
\exp\left(\frac{1}{\omega_n}\int_{\mathbb{S}^{n-1}}-\frac{\log h_K}{\varphi}d\sigma\right)\geq a &\mbox{if } p=0\\
\int_{\mathbb{S}^{n-1}}\frac{h_K^p}{\varphi}d\sigma\geq a &\mbox{if } p\in (-n,0).
\end{cases}
\]
\end{itemize}
\end{theorem}
\begin{proof}
The last set of claims follows easily:
$S_{e,\varphi,p,a}\subset \mathcal{K}_e^n$ and $\{d(K)\}_{K\in S_{e,\varphi,p,a}}$ is uniformly bounded by Lemma \ref{lem: bound diam}. Therefore, since volume is fixed, in-radii of convex bodies in $S_{e,\varphi,p,a}$ are uniformly bounded below. Moreover, for $K\in \mathcal{K}_e^n$, the ball with maximal radius enclosed by $K$ must be centered at the origin. Next we prove the remaining claims.
Since volume is normalized and $\{d(K)\}_{K\in S_{\varphi,p,a}}$ is uniformly bounded by Lemma \ref{lem: bound diam}, it is enough to prove that there exists  $r>0$ such that $h_K>r$ for any $K\in S_{\varphi,p,a}.$
Suppose, on the contrary, that there is a sequence of convex bodies $\{K_i\}\subset S_{\varphi,p,a}$ such that $\operatorname{dist}(\vec{o},\partial K_i)\to 0.$ By Lemma \ref{lem: bound diam}, $\{d(K_i)\}$ is uniformly bounded above. Thus, by Blaschke's selection theorem, $\{K_i\}$ converges (passing to a further subsequence if necessary) to $K_{\infty}$ in the Hausdorff distance and, additionally, $\vec{o}\in \partial K_{\infty}.$ On the other hand, by Theorems \ref{thm: cont entropy points} and \ref{thm: cont ent func}, $K_{\infty}\in S_{\varphi,p,a}.$ This is a contradiction.
\end{proof}
\begin{corollary}\label{cor: uniform lower and upper on support functions}
Under the assumptions of Theorems \ref{thm: 3a}, A1 with $p\neq -2$, A2, and A3, there exist $r,R$ such that
$0<r\leq h_{\tilde{K}_t}\leq R<\infty.$
\end{corollary}
\begin{proof}
By Lemma \ref{lem: 8}, entropy points remain at the origin. By Lemma \ref{lem: monotonicity}, $K_t$ for $t>0$ belongs to the same entropy class as $K_0$. Therefore, the claim follows from Theorem \ref{thm: uniform lower bound on the support function}.
\end{proof}
\begin{theorem}\label{thm: converg curvature image}
The following statements are true:
\begin{itemize}
  \item Let $-n\leq p<\infty$, and assume $\{K_i\}\subset \mathcal{K}_0^n$ with $e_p(K_i)=\vec{o}$ converges in the Hausdorff distance to $K_{\infty}.$
      Then $\{\Lambda_p K_i\}$ converges in the Hausdorff distance to $\Lambda_p K_{\infty}.$
    \item Let $-n\leq p<\infty,~\varphi\in\mathcal{S}^+_e$, and assume that $\{K_i\}\subset \mathcal{K}_e^n$ converges in the Hausdorff distance to $K_{\infty}.$
        Then $\{\Lambda_p^{\varphi} K_i\}$ converges in the Hausdorff distance to $\Lambda_p^{\varphi} K_{\infty}.$
  \item Let $-n\leq p\leq-n+1,~\varphi\in \mathcal{S}^+,$ and assume $\{K_i\}\subset \mathcal{K}_0^n$ with $e_p^{\varphi}(K_i)=\vec{o}$ converges in the Hausdorff distance to $K_{\infty}.$
      Then $\{\Lambda_p^{\varphi} K_i\}$ converges in the Hausdorff distance to $\Lambda_p^{\varphi} K_{\infty}.$
\end{itemize}
\end{theorem}
\begin{proof}
We give the proof of the first statement.
By Theorem \ref{thm: cont entropy points}, $\vec{o}=e_p(K_i)\to e_p(K_{\infty})$. Since $e_p(K_{\infty})=\vec{o}$ is an interior point of $K_{\infty}$, we have $r \leq h_{K_i}\leq R$ for some $0<r,R<\infty$ independent of $i$. On the other hand, in view of (\ref{def: p curvature}), $\operatorname{cent}(\Lambda_p K_i)=\operatorname{cent}(K_i)$, and \cite[Lemma 3, Lemma 4]{ChYau}, we conclude that there exist $0<r',R'<\infty$ independent of $i$, such that $\Lambda_p K_i\subset B_{R'}$ and $V(\Lambda_p K_i)\geq \omega_nr'^{n}$ (note that in both Lemmas 3, 4 of \cite{ChYau} the assumption that $K$ is of class $ C^4$ is unnecessary; therefore, here we do not need to know the regularity of $\Lambda_pK_i$). Take a convergent subsequence of $\{\Lambda_p K_i\}$ and denote it again by $\{\Lambda_p K_i\}$. The limiting figure must be a convex body, say $\tilde{K}.$ Choose an arbitrary convex body $P$. From the weak continuity of surface area measures we get
\[\lim_{i\to\infty}\int_{\mathbb{S}^{n-1}}h_PdS_{\Lambda_pK_i}=\int_{\mathbb{S}^{n-1}}h_PdS_{\tilde{K}}=V_1(\tilde{K},P).\]
Moreover,
\begin{align*}
\lim_{i\to\infty}\int_{\mathbb{S}^{n-1}}h_PdS_{\Lambda_pK_i}&=
\lim_{i\to\infty}\left(\frac{V(K_i)}{\frac{1}{n}\int_{\mathbb{S}^{n-1}}h_{K_i}^pd\sigma}\int_{\mathbb{S}^{n-1}}h_Ph_{K_i}^{p-1}d\sigma\right)\\
&=\frac{V(K_{\infty})}{\frac{1}{n}\int_{\mathbb{S}^{n-1}}h_{K_{\infty}}^pd\sigma}\int_{\mathbb{S}^{n-1}}h_Ph_{K_{\infty}}^{p-1}d\sigma\\
&=\int_{\mathbb{S}^{n-1}}h_PdS_{\Lambda_pK_{\infty}}=V_1(\Lambda_pK_{\infty},P).
\end{align*}
Since $V_1(\Lambda_pK_{\infty},P)=V_1(\tilde{K},P)$ holds for any convex body $P,$ we conclude that $\Lambda_pK_{\infty}$ is a translation of $\tilde{K}$; see \cite[Theorem 8.1.2]{Schneider}. Furthermore, note that $\operatorname{cent}(K_{i})=\operatorname{cent}(\Lambda_pK_{i}),$ thus $\operatorname{cent}(\Lambda_pK_{\infty})=\operatorname{cent}(K_{\infty})=\operatorname{cent}(\tilde{K}).$ That is, no translation is needed; $\Lambda_pK_{\infty}=\tilde{K}.$ Finally, notice that the limit is independent of the convergent subsequence. The proof is complete.
\end{proof}
\begin{lemma}\label{lem: large ent}
Under the assumptions of Theorems \ref{thm: 1}, \ref{thm: 3a}, A1, A2, and A3 we have
\[
\begin{cases}
\lim\limits_{t\to T}\frac{1}{p}\log \int_{\mathbb{S}^{n-1}}\frac{h_{K_t}^p}{\varphi}d\sigma=\infty &\mbox{if } p\ne 0\\
\lim\limits_{t\to T}\int_{\mathbb{S}^{n-1}}\frac{\log h_{K_t}}{\varphi}d\sigma=\infty & \mbox{if } p=0.
\end{cases}
\]
\end{lemma}
\begin{proof}
We prove the claim for $p\ne 0$ and $\varphi\equiv1$. First, we consider $p\ne -n.$
Corollary \ref{cor: uniform lower and upper on support functions} shows that $0<r\leq h_{\tilde{K}_t}\leq R<\infty$. Thus $\frac{\min h_{K_t}}{\max h_{K_t}}\geq \frac{r}{R}\Rightarrow\frac{1}{p}\log \int_{\mathbb{S}^{n-1}}h_{K_t}^pd\sigma\geq \frac{1}{p}\log \int_{\mathbb{S}^{n-1}}(\max h_{K_t})^pd\sigma+c.$
Since $\max h_{K_t}\to \infty$ by Proposition \ref{prop: expansion to infty} and  Lemma \ref{lem: expansion to inftya}, the claim follows.
Now we consider case $p=-n$. Since $V(K)\int_{\mathbb{S}^{n-1}}\frac{1}{h_{K}^{n}}d\sigma$ is $GL(n)$ invariant and $\mathcal{A}_{-n}(K_t)$ is monotone along the flow, we get
\begin{align*}
&\lim\limits_{t\to T}n\mathcal{A}_{-n}(K_t)\\
&=\lim\limits_{t\to T}V(K_t)\int_{\mathbb{S}^{n-1}}\frac{1}{h_{K_t}^{n}}d\sigma=\lim\limits_{t\to T}V(l_tK_t)\int_{\mathbb{S}^{n-1}}\frac{1}{h_{l_tK_t}^n}d\sigma\\
&\leq \limsup\limits_{t\to T}V(l_tK_t)\int_{\mathbb{S}^{n-1}}\frac{1}{h_{l_tK_t-p_t}^n}d\sigma.
\end{align*}
where $l_t\in GL(n)$ and $p_t\in\operatorname{int} l_tK_t.$ Note that to get this last inequality we used the fact that $\int_{\mathbb{S}^{n-1}}\frac{1}{h_{l_tK_t-p_t}^n}d\sigma$ is minimized on $l_tK_t$ only when $p_t$ is $e_{-n}(l_tK_t)=l_te_{-n}(K_t)=\vec{o}$ (see Petty \cite[Lemma 2.2]{Petty 1975} for a proof that the Santal\'{o} point mapping, $e_{-n}$, is affinely equivariant and continuous). Moreover, we can choose $l_t,~p_t$ such that $V(l_tK_t)=\omega_n$ and $0<r<h_{l_tK_t-p_t}<R<\infty$
for some universal constants $r,R.$ This ensures that the limit of $n\mathcal{A}_{-n}(K_t)$ along the flow is finite. That is, for some $c<\infty$ we have
\[V(K_t)\int_{\mathbb{S}^{n-1}}\frac{1}{h_{K_t}^{n}}d\sigma<c.\]
Now recall Proposition \ref{cor: volume to infty} that $\lim\limits_{t\to T}V(K_t)=\infty.$ Therefore,
\[\lim\limits_{t\to T}\int_{\mathbb{S}^{n-1}}\frac{1}{h_{K_t}^{n}}d\sigma=0.\]
\end{proof}
\begin{lemma}\label{lem: key}Under the assumptions of Theorems \ref{thm: 1}, A1, A2, and A3 we have for a subsequence of $\{K_t\}$ that
\[\lim_{{t_i}\to T}\frac{V_1(K_{t_i},\Lambda_p^{\varphi}K_{t_i})}{V(\Lambda_p^{\varphi}K_{t_i})}=1,~\lim_{{t_i}\to T}\frac{V_1(K_{t_i},\Lambda_pK_{t_i})}{V(\Lambda_p^{\varphi}K_{t_i})}=1.\]
\end{lemma}
\begin{proof}
We prove the claim for $p\ne 0$. Suppose, on the contrary, that there exist $\varepsilon>0$ and $t_0>0,$  such that for any $t>t_0$ we have $\frac{V_1(K_t,\Lambda_p^{\varphi}K_t)}{V(\Lambda_p^{\varphi}K_t)}-1\ge \varepsilon.$ From (\ref{eq: key asymp}) it follows that
\begin{equation}\label{eq: contradiction}
\frac{d}{dt}\frac{\log\mathcal{B}_p^{\varphi}(K_t)}{1-p}\geq n\varepsilon\frac{d}{dt}\left(\frac{1}{p}\log \int_{\mathbb{S}^{n-1}}\frac{h_{K_t}^p}{\varphi}d\sigma\right).
\end{equation}
On the other hand, by Lemma \ref{lem: large ent}, $\frac{1}{p}\log \int_{\mathbb{S}^{n-1}}\frac{h_{K_t}^p}{\varphi}d\sigma$ can be made arbitrarily large if $t$ is close enough to $T$. By integrating both sides of (\ref{eq: contradiction}) on $[t_0,s)$ and then sending $s\to T$ we get
\[
\lim_{t\to T}\frac{\log\mathcal{B}_p^{\varphi}(K_{t})}{1-p}=\infty.
\]
This is a contradiction in view of Theorem \ref{thm: converg curvature image}: First, note that $\mathcal{B}_p^{\varphi}$ is scaling-invariant. Second, when $p\ne -n$, Corollary \ref{cor: uniform lower and upper on support functions} shows that $0<r\leq h_{\tilde{K}_t}\leq R<\infty$.
Therefore, by Blaschke's selection theorem, there is a subsequence, $\{\tilde{K}_{t_i}\}$, that converges in the Hausdorff distance to a limiting shape, say $\tilde{K}_{\infty}.$ Taking Theorem \ref{thm: converg curvature image} into consideration, we conclude that $\{\Lambda_p^{\varphi}\tilde{K}_{t_i}\}$ converges to $\Lambda_p^{\varphi}K_{\infty}.$ Consequently,
\[\infty \leftarrow \frac{\log\mathcal{B}_p^{\varphi}(K_{t_i})}{1-p}=\frac{\log\mathcal{B}_p^{\varphi}(\tilde{K}_{t_i})}{1-p}\to \frac{\log\mathcal{B}_p^{\varphi}(\tilde{K}_{\infty})}{1-p}.\]
Since $0<r\leq h_{\tilde{K}_{\infty}}\leq R<\infty,$ we must have $0<\mathcal{B}_p^{\varphi}(\tilde{K}_{\infty})<\infty$. When $p=-n,$ note that $\frac{\log\mathcal{B}_p(K_{t})}{1-p}$ is $GL(n)$-invariant. Therefore, we may assume that $\{d(l_t\tilde{K}_t)\}$ is uniformly bounded for suitable choices of $l_t\in SL(n).$ We may now continue the previous argument for $p\ne -n,~\varphi\equiv 1$ to reach a contradiction.
\end{proof}
\section{Proofs of Theorems A1, A2, A3 with convergence in the $ C^1$-topology and Proofs of Theorems \ref{thm: 1}, \ref{thm: 3}}\label{sec: a1a2a3}
Proofs given in this section rely heavily on entropy functionals $B_p^{\varphi}.$ In this section we will not give the proof of Theorem \ref{thm: 3a}, since Lemma \ref{lem: key}, in view of $\frac{d}{dt}B_1^{\varphi}\equiv0$, is not useful when $p=1.$ See Section \ref{sec: before application} for the proof.

\begin{fact}\label{thm: key reg}
Let $\varphi$ be a positive function on the unit sphere of class $ C^{k,\alpha}$, where $k$ is a non-negative integer and $0<\alpha<1$. Assume $K\in \mathcal{K}_0^n$ has a positive continuous curvature function such that $f_Kh_K^{1-p}=\varphi$. Then $K$ is of class $ C^{k+2,\alpha}_{+}.$
\end{fact}
\begin{proof}
Since $h_K>0$, the Gauss curvature of $K$ in the generalized sense (in the sense of Alexandrov), and thus in the viscosity sense is pinched between two positive numbers. Hence, by Caffarelli \cite[Corollary 3]{Caffarelli 1990-1} $K$ is strictly convex (when $n=3$ this conclusion follows from Alexandrov's theorem 1942 \cite{Aleksandrov 1942}). This in turn implies that $h_K$ is $ C^1$ (cf. \cite[Corollary 1.7.3]{Schneider}), and so the right-hand side of $f_K=h_K^{p-1}\varphi$ is positive and $ C^{\alpha}$. Thus, $K$ is of class $ C^{2,\alpha}_{+};$ it follows from Caffarelli's work \cite[Theorem 4]{Caffarelli 1990-2} that, for given nonnegative integer $k$ and $0<\alpha<1$, the solution of Minkowski's problem is of class $ C^{k+2,\alpha}$ if the given Gauss curvature on the spherical image is of class $ C^{k,\alpha}$ (see also Jerison \cite[Theorem 0.7]{Jerison}). So $h_K$ is $ C^2.$ Higher order regularity up to $ C^{k+2,\alpha}_{+}$ follows by induction using \cite[Theorem 4]{Caffarelli 1990-2}.
\end{proof}
\subsection{Proofs of Theorems A1, A2, A3 with convergence in the $ C^1$-topology}\label{sec: ref to stab}
We prove the statements for $p\ne0,-2.$ We will consider the case $p=-2$ in section \ref{sec: thm1thm3}.
Consider the sequence introduced in Lemma \ref{lem: key}.
Corollary \ref{cor: uniform lower and upper on support functions} shows that a subsequence of $\{\tilde{K}_{t_i}\}$ converges in the Hausdorff distance to a limiting shape $\tilde{K}_{\infty}$ with the origin in its interior. By Theorem \ref{thm: converg curvature image} and Lemma \ref{lem: key} we conclude that
\[V_1(\tilde{K}_{\infty},\Lambda^{\varphi}_p\tilde{K}_{\infty})=V(\Lambda^{\varphi}_p\tilde{K}_{\infty})\Rightarrow \tilde{K}_{\infty}=\Lambda_p^{\varphi}\tilde{K}_{\infty}.\]
In particular, we conclude that $\lim\limits_{t_i\to T}\frac{V(\Lambda_p^{\varphi} \tilde{K}_{t_i})}{V(\tilde{K}_{t_i})}=1.$ Therefore,
$0<\lim\limits_{t_i\to T}\mathcal{A}_p^{\varphi}(\tilde{K}_{t_i})=\lim\limits_{t_i\to T}\mathcal{B}_p^{\varphi}(\tilde{K}_{t_i})<\infty.$ So
monotonicity of $\mathcal{A}_p^{\varphi}(\tilde{K}_t)$ and $\mathcal{B}_p^{\varphi}(\tilde{K}_t)$ yield $\lim\limits_{t\to T}\mathcal{A}_p^{\varphi}(\tilde{K}_{t})=\lim\limits_{t\to T}\mathcal{B}_p^{\varphi}(\tilde{K}_{t}).$ This, in turn, has two implications:
\begin{align}\label{key: lim key}
\lim\limits_{t\to T} \left(\frac{V(\Lambda_p^{\varphi} \tilde{K}_t)}{V(\tilde{K}_t)}\right)^{n-1}&=\lim\limits_{t\to T} \frac{\mathcal{A}_p^{\varphi}(\tilde{K}_t)}{\mathcal{B}_p^{\varphi}(\tilde{K}_{t})}=1,\\
\lim\limits_{t\to T}\int_{\mathbb{S}^{n-1}}\frac{(h_{\tilde{K}_t}(u)-e_p\cdot u)^p}{\varphi(u)}d\sigma&=\lim\limits_{t\to T}\left(\frac{\mathcal{A}_p^{\varphi}(\tilde{K}_t)}{V(\tilde{K}_t)}\right)^{-\frac{p}{n}}=\frac{\omega_n}{c}>0.\label{key: limkey2}
\end{align}
Corollary \ref{cor: uniform lower and upper on support functions} and equalities (\ref{key: lim key}) and (\ref{key: limkey2}) imply that every given subsequence of $\{\tilde{K}_t\}$ has a convergent subsequence such that its limit satisfies $$V(L)=V(\Lambda_p^{\varphi}L)\Rightarrow L=\Lambda_p^{\varphi}L,$$ and thus $L$ is a solution of
$${\varphi}h_{L}^{1-p}dS_{L}=\frac{\omega_n}{\lim\limits_{t\to T}\int_{\mathbb{S}^{n-1}}\frac{(h_{\tilde{K}_t}(u)-e_p\cdot u)^p}{\varphi(u)}d\sigma}d\sigma=cd\sigma.$$ Fact \ref{thm: key reg} implies that $L\in \mathcal{F}_0^n.$ The $ C^{1}$-convergence, which is purely geometric and does not depend on the evolution equation, follows from \cite[Lemma 13]{Andrews 1997}. Finally, when $p\geq1$ and $ p\neq n$, in view of the uniqueness theorem of Lutwak \cite[Corollary 2.3, $i$=0]{Lutwa 1993}, there is only one solution to ${\varphi}h_L^{1-p}dS_L=cd\sigma$ in $\mathbb{R}^n$ with volume $\omega_n$. In particular, if $\varphi\equiv1$, then $L$ must be the unit ball. When $p=n$ the uniqueness follows from \cite[Theorem B]{Chou Wang}; therefore, if $\varphi\equiv1$ and $V(L)=\omega_n$, then $L$ must be the unit ball. In $\mathbb{R}^2$ a classification result of Andrews \cite{Andrews 2003} states that for $-2<p<1,~\varphi\equiv1$ the only solution with area $\pi$ is the unit disk.
\subsection{Proofs of Theorems \ref{thm: 1}, \ref{thm: 3}}\label{sec: thm1thm3}
A convex body $K$ has its centroid at the origin if and only if $K^{\ast}$ has its Santal\'{o} point at the origin \cite[p. 546]{Schneider}. Thus, in view of Lemma \ref{lem: ev polar}, there is a one-one correspondence between solutions of (\ref{eq: flow4}) for $p=-n$ and solutions of (\ref{eq: flow3}). In view of John's ellipsoid lemma, there exist $l_{t_i}\in SL(n),$ such that $d(l_{t_i}\tilde{K}_{t_i})\le 2n.$ Since $e_{-n}(l_{t_i}\tilde{K}_{t_i})=\vec{o},$ by the Blaschke selection theorem and by Theorem \ref{thm: cont entropy points} we can show that, after passing to a subsequence, $l_{t_i}\tilde{K}_{t_i}\to \tilde{K}_{\infty}\in\mathcal{K}^n_0.$ Moreover, arguing as in Section \ref{sec: ref to stab}, $\tilde{K}_{\infty}$ is a weak solution of $h_K^{1+n}dS_K=cd\sigma$. It was proved by Philippis and Marini \cite{Philippis Marini}, and Schneider \cite[Theorem 10.5.1]{Schneider} that the origin-centered ellipsoids are the only solutions of $h_K^{1+n}dS_K=cd\sigma$ (this also follows from Fact \ref{thm: key reg} and the classical theorem of Pogorelov \cite[Theorem 4.3.1]{Gutierez})\footnote{In \cite[Lemma 8.1, Lemma 8.7]{Petty}, Petty proves that the origin-centered ellipsoids are the only solutions of $h_K^{1+n}dS_K=cd\sigma,$ provided either $n=2$, or $n\geq3$ and $K$ is rotationally symmetric.}. Therefore, by Theorem \ref{thm: cont ent func}, we get $\mathcal{A}_{-n}(K_{t_i})=\mathcal{A}_{-n}(l_{t_i}K_{t_i})=\mathcal{A}_{-n}(l_{t_i}\tilde{K}_{t_i})\to n\omega_n^2.$ Monotonicity of $\mathcal{A}_{-n}(K_t)$ then implies that $\lim_{t\to T}\mathcal{A}_{-n}(\tilde{K}_t)=n\omega_n^2.$ Consequently, for any arbitrary $\varepsilon>0$, we can choose a $t_{\varepsilon}$ large enough that $V(K_t)V(K_t^{\ast})=\frac{1}{n}\mathcal{A}_{-n}(K_t)\ge\omega_n^2/(1+\varepsilon)$ for all $t\in [t_\varepsilon,T).$ Choosing a small enough $\varepsilon>0$ ensures that the argument of \cite[Section 6]{Ivaki 2013} or \cite{Ivaki 2014} can be employed to prove the existence of a sequence of times $\{t_k\}$ approaching $T$ and a sequence of special linear transformations $l_{t_k}$ such that $\{l_{t_k}\tilde{K}_{t_k}\}$ converges in $C^{\infty}$ to the unit ball. The stronger convergence statements in the theorems follow from \cite[Proposition 9.2.4]{Lu} by considering the linearization of the evolution equations about the space of origin-symmetric ellipsoids.
\section{Convergence in the $ C^{\infty}$-topology in Theorems \ref{thm: 3a}, A1, A2, A3}\label{sec: before application}
In this section we prove the $C^{\infty}$ convergence in Theorems \ref{thm: 3a}, A1, A2, and A3. Since we discussed the case $p=-2$ in Section \ref{sec: thm1thm3}, here we mainly focus on $p>-2.$ To prove the $ C^{\infty}$ convergence, we will only need to obtain a uniform upper bound for the Gauss curvature of the normalized solution. In fact, in Corollary \ref{cor: uniform lower and upper on support functions} we have established the first order regularity estimate $r\leq h_{\tilde{K}_t}\leq R$. Moreover, the lower bound on the Gauss curvature given in Lemma \ref{lem: upper}, $\mathcal{K}\geq 1/(a+b t^{-\frac{n-1}{n}})$, is independent of the initial data on $[t_0/2,t_0];$ it is quite standard that such a bound yields a uniform lower bound on the Gauss curvature of the normalized solution. Once we are equipped with uniform $C^k$ estimates for $\{\tilde{K}_t\}$, we can use the monotonicity of functionals $\mathcal{A}_p^{\varphi}$ to prove that there exists a sequence of times $\{t_k\}$ approaching $T$ for which $\{\tilde{K}_{t_k}\}$ converges in $C^{\infty}$ to a solution of (\ref{def: self similar}). The stronger convergence statement in Theorem \ref{thm: 3a} follows from the uniqueness of the self-similar solution in the class of origin-symmetric convex bodies.
\begin{lemma}\label{lem: more ev for reg}
The following evolution equations hold along the flow (\ref{eq: flow4}):
\begin{align*}
\partial_t |F|^{1+n-p}&=(F\cdot \nu)^{2-p}\varphi \frac{\mathcal{\dot{K}}^{ij}}{\mathcal{K}^2}\nabla_i\nabla_j|F|^{1+n-p}\\
&+\varphi n(1+n-p)\frac{(F\cdot \nu)^{3-p}}{\mathcal{K}}|F|^{-1+n-p}\\
&-\varphi (1+n-p)(F\cdot \nu)^{2-p} \frac{\mathcal{\dot{K}}^{ij}}{\mathcal{K}^2}|F|^{-1+n-p}g_{ij}\\
&-\varphi(-1+n-p)(1+n-p)(F\cdot \nu)^{2-p} |F|^{-3+n-p}\frac{\mathcal{\dot{K}}^{ij}}{\mathcal{K}^2}(F_i\cdot F)(F_j\cdot F),
\end{align*}
and
\begin{align*}
\partial_t \left(\varphi\frac{(F\cdot\nu)^{2-p}}{\mathcal{K}}\right)=&(F\cdot \nu)^{2-p}\varphi \frac{\mathcal{\dot{K}}^{ij}}{\mathcal{K}^2}\nabla_i\nabla_j\left(\varphi\frac{(F\cdot\nu)^{2-p}}{\mathcal{K}}\right)\\
&+(F\cdot\nu)^{4-2p}\varphi^2\frac{\dot{\mathcal{K}}_j^i}{\mathcal{K}^3}w_i^kw_k^j\\
&+\varphi^2(2-p)\frac{(F\cdot\nu)^{3-2p}}{\mathcal{K}^2}\\
&-\varphi(2-p)\frac{(F\cdot\nu)^{1-p}}{\mathcal{K}}F\cdot DF\left(\nabla\left(\varphi\frac{(F\cdot\nu)^{2-p}}{\mathcal{K}}\right)\right)\\
&-\frac{(F\cdot\nu)^{2-p}}{\mathcal{K}}d\varphi\left( DF\left(\nabla\left(\varphi\frac{(F\cdot\nu)^{2-p}}{\mathcal{K}}\right)\right)\right).
\end{align*}
\end{lemma}
\begin{proof} We calculate
\begin{align*}
&-(F\cdot \nu)^{2-p}\varphi \frac{\mathcal{\dot{K}}^{ij}}{\mathcal{K}^2}\nabla_i\nabla_j|F|^{1+n-p}=\\
&+\varphi (n-1)(1+n-p)\frac{(F\cdot \nu)^{3-p}}{\mathcal{K}}|F|^{-1+n-p}\\
&-\varphi(1+n-p)(F\cdot \nu)^{2-p} |F|^{-1+n-p}\frac{\mathcal{\dot{K}}^{ij}}{\mathcal{K}^2}g_{ij}\\
&-\varphi(-1+n-p)(1+n-p)(F\cdot \nu)^{2-p} |F|^{-3+n-p}\frac{\mathcal{\dot{K}}^{ij}}{\mathcal{K}^2}(F_i\cdot F)(F_j\cdot F).
\end{align*}
On the other hand,
$$\partial_t |F|^{1+n-p}=\varphi(1+n-p)\frac{(F\cdot \nu)^{3-p}}{\mathcal{K}}|F|^{-1+n-p}.$$
Therefore,
\begin{align*}
\partial_t |F|^{1+n-p}&=(F\cdot \nu)^{2-p}\varphi \frac{\mathcal{\dot{K}}^{ij}}{\mathcal{K}^2}\nabla_i\nabla_j|F|^{1+n-p}\\
&+\varphi n(1+n-p)\frac{(F\cdot \nu)^{3-p}}{\mathcal{K}}|F|^{-1+n-p}\\
&-\varphi (1+n-p)(F\cdot \nu)^{2-p} \frac{\mathcal{\dot{K}}^{ij}}{\mathcal{K}^2}|F|^{-1+n-p}g_{ij}\\
&-\varphi(-1+n-p)(1+n-p)(F\cdot \nu)^{2-p} |F|^{-3+n-p}\frac{\mathcal{\dot{K}}^{ij}}{\mathcal{K}^2}(F_i\cdot F)(F_j\cdot F).
\end{align*}
To calculate the evolution equation of $\varphi\frac{(F\cdot\nu)^{2-p}}{\mathcal{K}}$, we will employ the following two evolution equations (which can be obtained with straightforward computations; for example, see Andrews \cite[Theorem 3.7]{Andrews 1994}):
\begin{align*}
\partial_t w_i^j=-\nabla_i\nabla^j\left(\varphi\frac{(F\cdot\nu)^{2-p}}{\mathcal{K}}\right)-\varphi\frac{(F\cdot\nu)^{2-p}}{\mathcal{K}}w_i^kw_k^j,
\end{align*}
\begin{align*}
\partial_t \nu=-DF\left(\nabla\left(\varphi\frac{(F\cdot\nu)^{2-p}}{\mathcal{K}}\right)\right).
\end{align*}
Therefore,
\begin{align*}
\partial_t \left(\varphi\frac{(F\cdot\nu)^{2-p}}{\mathcal{K}}\right)=&(F\cdot \nu)^{2-p}\varphi \frac{\mathcal{\dot{K}}^{ij}}{\mathcal{K}^2}\nabla_i\nabla_j\left(\varphi\frac{(F\cdot\nu)^{2-p}}{\mathcal{K}}\right)\\
&+(F\cdot\nu)^{4-2p}\varphi^2\frac{\dot{\mathcal{K}}_j^i}{\mathcal{K}^3}w_i^kw_k^j\\
&+\varphi^2(2-p)\frac{(F\cdot\nu)^{3-2p}}{\mathcal{K}^2}\\
&-\varphi(2-p)\frac{(F\cdot\nu)^{1-p}}{\mathcal{K}} F\cdot DF\left(\nabla\left(\varphi\frac{(F\cdot\nu)^{2-p}}{\mathcal{K}}\right)\right)\\
&-\frac{(F\cdot\nu)^{2-p}}{\mathcal{K}}d\varphi\left( DF\left(\nabla\left(\varphi\frac{(F\cdot\nu)^{2-p}}{\mathcal{K}}\right)\right)\right).
\end{align*}
\end{proof}
\begin{lemma}\label{lem: high reg}
Assume that $n=2$ and $-\infty<p<\infty$ or $n\geq 3$ and $p\leq n.$
Suppose there exists $0<\gamma<1$ such that solution $\{K_t\}$ to (\ref{eq: flow4}) satisfies $\gamma|F|\leq F\cdot\nu $ on $[0,t_0].$ Then there exists $\lambda>0$ (independent of $t_0$) such that $\chi(\cdot,t):=|F|^{1+n-p}(\cdot,t)-\lambda \varphi\frac{(F\cdot\nu)^{2-p}}{\mathcal{K}}(\cdot,t)$ is always negative on $[0,t_0].$
\end{lemma}
\begin{proof}
Take $\lambda$ such that $\chi$ is negative at time $t=0.$ We will prove that $\chi$ remains negative, perhaps for a larger value of $\lambda$. We calculate the evolution equation of $\chi$ and apply the maximum principle to $\chi$ on $[0,\tau]$, where $\tau>0$ is the first time that for some $y\in\partial K_{\tau}$ we have $\chi(y,\tau)=0$. Notice that at such a point, where the maximum of $\chi$ is achieved, we have:
\begin{align*}\nabla \chi=0\Rightarrow DF\left(\nabla\left(\lambda \varphi\frac{(F\cdot\nu)^{2-p}}{\mathcal{K}}\right)\right)&= DF\left(\nabla|F|^{1+n-p}\right)\\
&=(1+n-p)|F|^{-1+n-p}F^{\top},
\end{align*}
where $F^{\top}(\cdot,t)$ is the tangential component of $F(\cdot,t)$ to $\partial K_t.$ Furthermore,
 $$(F\cdot\nu)^{2-p}\frac{\dot{\mathcal{K}}^{ij}}{\mathcal{K}^{2}}\nabla_i\nabla_j \chi\leq0,$$
 and in view of the assumption $\gamma|F|\leq F\cdot\nu\leq |F|$ we get
\begin{align}\label{ie: bounding norm of F}
\gamma^{|2-p|}\left(\frac{\lambda\varphi}{\mathcal{K}}\right)^{\frac{1}{{n-1}}}\leq|F|\leq \gamma^{-|2-p|}\left(\frac{\lambda \varphi}{\mathcal{K}}\right)^{\frac{1}{{n-1}}}.
\end{align}
Also, using $(F_i\cdot F)(F_j\cdot F)\leq g_{ij}|F|^2$ and $p\leq n$, we may calculate
\begin{align*}
-&\varphi (1+n-p)(F\cdot \nu)^{2-p} |F|^{-1+n-p}\frac{\mathcal{\dot{K}}^{ij}}{\mathcal{K}^2}g_{ij}\\
-&\varphi(-1+n-p)(1+n-p)(F\cdot \nu)^{2-p} |F|^{-3+n-p}\frac{\mathcal{\dot{K}}^{ij}}{\mathcal{K}^2}(F_i\cdot F)(F_j\cdot F)\\
\leq& -\varphi(n-p)(1+n-p)(F\cdot \nu)^{2-p} |F|^{-3+n-p}\frac{\mathcal{\dot{K}}^{ij}}{\mathcal{K}^2}(F_i\cdot F)(F_j\cdot F)\leq0.
\end{align*}
Therefore, from Lemma \ref{lem: more ev for reg} it follows that at $(y,\tau)$ we have
\begin{align*}
\partial_t \chi&\leq \varphi n(1+n-p)\frac{(F\cdot \nu)^{3-p}}{\mathcal{K}}|F|^{-1+n-p}\\
&-\lambda(F\cdot\nu)^{4-2p}\varphi^2\frac{\dot{\mathcal{K}}_j^i}{\mathcal{K}^3}w_i^kw_k^j-\lambda\varphi^2(2-p)\frac{(F\cdot\nu)^{3-2p}}{\mathcal{K}^2}\\
&+\varphi(2-p)(1+n-p)\frac{(F\cdot\nu)^{1-p}}{\mathcal{K}}|F|^{-1+n-p}|F^{\top}|^2\\
&+(1+n-p)\frac{(F\cdot\nu)^{2-p}}{\mathcal{K}}|F|^{n-p}|\bar{\nabla}\varphi|\\
&\leq \varphi n(1+n-p)\frac{(F\cdot \nu)^{3-p}}{\mathcal{K}}|F|^{-1+n-p}\\
&-\lambda\varphi^2(n-1)(F\cdot\nu)^{4-2p}\mathcal{K}^{\frac{n}{n-1}-3}-\lambda\varphi^2(2-p)\frac{(F\cdot\nu)^{3-2p}}{\mathcal{K}^2}\\
&+\varphi(2-p)(1+n-p)\frac{(F\cdot\nu)^{1-p}}{\mathcal{K}}|F|^{-1+n-p}|F^{\top}|^2\\
&+(1+n-p)\frac{(F\cdot\nu)^{2-p}}{\mathcal{K}}|F|^{n-p}|\bar{\nabla}\varphi|.
\end{align*}
Here to obtain the second inequality we used the inverse-concavity of $\mathcal{K}^{\frac{1}{n-1}}: \dot{\mathcal{K}}^i_{j}w_i^kw_k^j\geq (n-1)\mathcal{K}^{\frac{n}{n-1}}.$ Now using inequalities $\gamma|F|\leq F\cdot\nu\leq |F|$ and (\ref{ie: bounding norm of F}), we obtain
\[0\le\partial_t\chi \leq \frac{\lambda^{\frac{1+n-2p}{n-1}}}{\mathcal{K}^{\frac{2n-2p+1}{n-1}}}\left(c+b\lambda^{\frac{1}{n-1}}-a\lambda^{\frac{2}{n-1}}\right)\]
for some $a,b,c>0$ depending on $p,\gamma,\varphi.$\footnote{The $\lambda$-term with the highest exponent stems from $-\lambda(F\cdot\nu)^{4-2p}\varphi^2\frac{\dot{\mathcal{K}}_j^i}{\mathcal{K}^3}w_i^kw_k^j$ or equivalently $-\lambda\varphi^2(n-1)(F\cdot\nu)^{4-2p}\mathcal{K}^{\frac{n}{n-1}-3}.$} Therefore, taking $\lambda$ large enough proves the claim.

Next we consider the case $n=2.$ We observe that
\begin{align*}
-&\varphi(-1+n-p)(1+n-p)(F\cdot \nu)^{2-p} |F|^{-3+n-p}\frac{\mathcal{\dot{K}}^{ij}}{\mathcal{K}^2}(F_i\cdot F)(F_j\cdot F)\\
\leq&|\varphi(1-p)(3-p)|(F\cdot \nu)^{2-p}\frac{|F|^{1-p}}{\mathcal{K}^2},
\end{align*}
 Therefore,
\begin{align*}
0\leq\partial_t \chi&\leq 2(3-p)\varphi \frac{(F\cdot \nu)^{3-p}}{\mathcal{K}}|F|^{1-p}\\
&-\varphi(3-p)(F\cdot \nu)^{2-p}|F|^{1-p} \frac{1}{\mathcal{K}^2}\\
&+|\varphi(1-p)(3-p)|(F\cdot \nu)^{2-p} |F|^{1-p}\frac{1}{\mathcal{K}^2}\\
&-\lambda\varphi^2\frac{(F\cdot\nu)^{4-2p}}{\mathcal{K}}-\lambda\varphi^2(2-p)\frac{(F\cdot\nu)^{3-2p}}{\mathcal{K}^2}\\
&+\varphi(2-p)(3-p)\frac{(F\cdot\nu)^{1-p}}{\mathcal{K}}|F|^{1-p}|F^{\top}|^2\\
&+|3-p|\frac{(F\cdot\nu)^{2-p}}{\mathcal{K}}|F|^{2-p}|\bar{\nabla}\varphi|\\
&\leq \frac{\lambda^{3-2p}}{\mathcal{K}^{5-2p}}\left(c+b\lambda-a\lambda^2\right)
\end{align*}
for some $a,b,c>0$ depending on $\gamma,\varphi.$ Thus, taking $\lambda$ large enough proves the claim.
\end{proof}
The next corollary and our discussion in the beginning of this section complete our argument for deducing the $C^{\infty}$ convergence in Theorem \ref{thm: 3a}.
\begin{corollary}
Suppose $p>-2.$ Under the assumptions of Theorems \ref{thm: 3a}, A1, A2, and A3, Gauss curvature of the normalized solution is uniformly bounded above. That is,
$$\left(\frac{V(K_t)}{V(B)}\right)^{\frac{n-1}{n}}\mathcal{K}(\cdot,t)\leq C<\infty$$
for some $C$ depending only on $K_0,~p,~\varphi.$
\end{corollary}
\begin{proof}
Corollary \ref{cor: uniform lower and upper on support functions} guarantees that the assumption of Lemma \ref{lem: high reg} is satisfied. Since the degrees of homogeneity of $|F|^{1+n-p}$ and $\lambda \varphi\frac{(F\cdot\nu)^{2-p}}{\mathcal{K}}$ are equal, we conclude the upper bound for the normalized Gauss curvature.
\end{proof}
\section{Applications}\label{sec: application}
\begin{theorem}[Lutwak \cite{Lutwak 1986}]\label{thm: Lutwak 1986} For any convex body $K$ we have
\[V(K)V((K-e_{-n})^{\ast})\leq \left(\frac{V(\Lambda_{-n}K)}{V(K)}\right)^{n-1}\omega_n^2\le \omega_n^2.\]
\end{theorem}
\begin{proof}
We may first prove the claim for $ C_+^{\infty}$ convex bodies. The general case follows from the first part of Theorem \ref{thm: converg curvature image} and a standard approximation argument \cite[Section 3.4]{Schneider}. Since the inequality is translation invariant, we may assume $e_{-n}(K)=\vec{o}.$ We employ
(\ref{eq: flow4}) with $p=-n$ and initial data $K_0:=K.$ Now the claim follows from Lemma \ref{lem: 8}, the monotonicity of $\mathcal{B}_{-n}(K_t)$ established in Lemma \ref{lem: monotonicity}, and Theorem \ref{thm: 1}.
\end{proof}
\begin{theorem}\label{thm: sharp inequlities}
For any convex body $K\in\mathcal{K}^2$ we have
\[
\begin{cases}
V(K)\left(\int_{\mathbb{S}^{1}}(h_K(u)-e_p\cdot u)^pd\sigma\right)^{-\frac{2}{p}}\geq \pi(2\pi)^{\frac{-2}{p}}\frac{V(\Lambda_pK)}{V(K)} &\mbox{if } p>1\\
V(K)\left(\int_{\mathbb{S}^{1}}(h_K(u)-e_p\cdot u)^pd\sigma\right)^{-\frac{2}{p}}\leq \pi(2\pi)^{\frac{-2}{p}}\frac{V(\Lambda_pK)}{V(K)} &\mbox{if } -2<p<1~\&~p\neq0\\
V(K)\exp\left(\frac{\int_{\mathbb{S}^{1}}-\log (h_K(u)-e_0\cdot u)d\sigma}{\pi}\right)\leq \pi\frac{V(\Lambda_0K)}{V(K)} & \mbox{if } p=0.
\end{cases}
\]
\end{theorem}
\begin{proof}
The claims follow from Lemmas \ref{lem: 8}, \ref{lem: monotonicity} and from Theorem A1.
\end{proof}
As we have already mentioned, for $p>1,$ $\Lambda_p K=K$ implies that $K$ is a ball (cf. Section \ref{sec: ref to stab}). In the remaining of this section, we give a stability version of this fact in $\mathbb{R}^2$: if $\frac{V(K)}{V(\Lambda_p K)}$ is close to one, then $K$ is close to a disk in the Hausdorff distance. To this end, we first recall Urysohn's inequality. Let us denote the mean width of $K\in\mathcal{K}^n$ by $w(K)=\frac{2}{n\omega_n}\int_{\mathbb{S}^{n-1}}h_Kd\sigma.$ Urysohn's inequality states that
\[\frac{V(K)}{\left(\frac{w(K)}{2}\right)^n}\leq \omega_n,\]
and equality holds exclusively for balls. The next lemma gives a lower bound for this ratio.
\begin{lemma}\label{thm: new}
Assume that $p>1$. For $K\in\mathcal{K}^2$ we have
\[\pi\frac{V(\Lambda_p K)}{V(K)}\le\frac{V(K)}{\left(\frac{w(K)}{2}\right)^2}.\]
\end{lemma}
\begin{proof}
Theorem \ref{thm: sharp inequlities} gives
\[\pi(2\pi)^{-\frac{2}{p}}\frac{V(\Lambda_pK)}{V(K)}\leq V(K)\left(\int_{\mathbb{S}^{1}}(h_K(u)-e_p\cdot u)^pd\sigma\right)^{-\frac{2}{p}}.\]
Applying the H\"{o}lder inequality to the right-hand side completes the proof.
\end{proof}
\begin{theorem}\label{thm: application}
Assume that $K\in\mathcal{K}^2$ and $p> 1.$ There exist $\gamma, \varepsilon_0>0$ with the following property. If $\frac{V(K)}{V(\Lambda_p K)}\le 1+\varepsilon$ for $\varepsilon\leq\varepsilon_0$, then there exist $x\in\operatorname{int}K$ and an origin-centered disk $\bar{B}$ such that
\[d_{\mathcal{H}}\left(\left(\frac{\pi}{V(K)}\right)^{\frac{1}{2}}(K-x), \bar{B}\right)\leq \gamma\varepsilon^{\frac{1}{3}}.\]
\end{theorem}
\begin{proof}
We may assume $V(K)=\pi.$ From Lemma \ref{thm: new} we get $w\leq 2(1+\varepsilon)^{\frac{1}{2}}.$  On the other hand, by \cite[Inequality 7.31, p. 385]{Schneider}
\[\left(\frac{w}{2}\right)^{\frac{1}{2}}-1\geq \left(\frac{w}{2}-\frac{1}{\rho_+}\right)^2\geq\left(1-\frac{1}{\rho_+}\right)^2.\]
Consequently,
\[\varepsilon^{\frac{1}{2}}\ge 1-\frac{1}{\rho+} \Rightarrow \rho_+\leq \frac{1}{1-\varepsilon^{\frac{1}{2}}}.\]
Suppose the Hausdorff distance of $K$ and a ball $\bar{B}$, $K\subset \bar{B}$, of radius $\frac{1}{1-\varepsilon^{\frac{1}{n}}}$ and center $x\in\operatorname{int}K$ is $d.$ Then, the volume of the hyperspherical cap of height $d$ inside $\bar{B}$ is given by
\[V_{\operatorname{cap}} = \frac{\pi^{\frac{n-1}{2}}\, \left(\frac{1}{1-\varepsilon^{\frac{1}{n}}}\right)^{n}}{\,\Gamma \left ( \frac{n+1}{2} \right )} \int\limits_{0}^{\arccos\left(1-d(1-\varepsilon^{\frac{1}{n}})\right)}\sin^n (t) \,dt.\]
Therefore,
\begin{align*}
& \frac{\pi ^ {\frac{n-1}{2}}}{\,\Gamma \left ( \frac{n+1}{2} \right )} \int\limits_{0}^{\frac{\pi}{2}}\sin^n (t) \,dt=V(K)\\
&\leq V(\bar{B})-V_{\operatorname{cap}}\\
&= \frac{\pi ^ {\frac{n-1}{2}}\, \left(\frac{1}{1-\varepsilon^{\frac{1}{n}}}\right)^{n}}{\,\Gamma \left ( \frac{n+1}{2} \right )}\left( \int\limits_{0}^{\frac{\pi}{2}}\sin^n (t) \,dt-\gamma_n'\left(\arccos\left(1-d(1-\varepsilon^{\frac{1}{n}})\right)\right)^{n+1}+\mbox{lower~ order~ terms}\right)\\
&\leq \frac{\pi ^ {\frac{n-1}{2}}\, \left(\frac{1}{1-\varepsilon^{\frac{1}{n}}}\right)^{n}}{\,\Gamma \left ( \frac{n+1}{2} \right )}\left(\int\limits_{0}^{\frac{\pi}{2}}\sin^n (t) \,dt-\gamma_n'2^{\frac{n+1}{2}}\left(d(1-\varepsilon^{\frac{1}{n}})\right)^{\frac{n+1}{2}}\right).
\end{align*}
This proves that the Hausdorff distance of $K-x$ from $\bar{B}-x$ is bounded above by $\gamma \varepsilon^{\frac{1}{3}}$ for some universal constant $\gamma>0.$
\end{proof}
\begin{noteadded}
Theorem \ref{thm: sharp inequlities} has been extended to higher dimensions in recent work of Andrews-Guan-Ni by using the classical affine isoperimetric inequality and isoperimetric inequality; see \cite{Andrews 2015}.
\end{noteadded}
\textbf{Acknowledgment}: I would like to thank Julian Scheuer for pointing out some errors in the proof of Lemma 9.2. I would like to thank the referees for useful comments and suggestions.
The work of the author was supported by Austrian Science Fund (FWF) Project M1716-N25  and the European Research Council (ERC) Project 306445.
\bibliographystyle{amsplain}

\end{document}